\documentclass[SIADSpaper]{siamonline1116}

\usepackage{amssymb,amsfonts,bbold}
\usepackage{graphicx}
\usepackage{epstopdf}
\usepackage{algorithmic}
\usepackage{lipsum}
\ifpdf
  \DeclareGraphicsExtensions{.eps,.pdf,.png,.jpg}
\else
  \DeclareGraphicsExtensions{.eps}
\fi

\numberwithin{theorem}{section}

\newcommand{\f}[2]{\frac{#1}{#2}}
\newcommand{\pd}[2]{\f{\partial #1}{\partial #2}}
\renewcommand{\d}[2]{\frac{d #1}{d #2}} 
\newsiamthm{assumption}{Assumption}
\newsiamremark{remark}{Remark}
\newsiamremark{example}{Example}

\title{\vspace{18pt} Hausdorff Continuity of Region
  of Attraction Boundary Under Parameter Variation with
  Application to Disturbance Recovery}

\author{Michael W. Fisher and Ian A. Hiskens}

\begin{document}

\maketitle

\begin{abstract}
Consider a parameter dependent vector field on either Euclidean space or a
compact Riemannian manifold.
Suppose that it possesses a parameter dependent initial condition
and a parameter dependent stable hyperbolic equilibrium point.
It is valuable to determine the set of parameter values, which we call the
recovery set, whose corresponding
initial conditions lie within the region of attraction of the corresponding
stable equilibrium point.
A boundary parameter value is a parameter value whose corresponding
initial condition lies in the boundary of the region of attraction of the
corresponding stable equilibrium point.
Prior algorithms numerically estimated the recovery set by estimating
its boundary via computation of boundary parameter values.
The primary purpose of this work is to provide theoretical justification for
those algorithms for a large class of parameter dependent vector fields.
This includes proving that, for these vector fields, the boundary of
the recovery set consists of boundary parameter values, and that the
properties exploited by the algorithms to compute these desired boundary
parameters will be satisfied.
The main technical result which these proofs rely on is establishing that
the region of attraction boundary varies continuously in an appropriate
sense with respect to small variation in parameter value for this class
of vector fields.
Hence, the majority of this work is devoted to proving this result,
which may be of independent interest.
The proof of continuity proceeds by proving that, for this class of vector
fields, the region of attraction permits a decomposition into a union of the
stable manifolds of the equilibrium points and periodic orbits it contains,
and this decomposition persists under small perturbations to the vector field.
\end{abstract}

\section{Introduction}


This work is motivated by physical and engineered systems that
possess a stable equilibrium point (SEP) representing
desired operation, and a parameter-dependent initial condition (IC)
which represents a parametrized, finite time disturbance.
As an example, consider a power
system subject to a lightning strike on a particular transmission line.
In such applications, it is important to understand whether the
system will be able to recover from the disturbance to the desired SEP.
This setting is well described by a parameter dependent
vector field, on either Euclidean space or a compact Riemannian manifold,
possessing a parameter dependent IC.
The IC of interest is the system state when the disturbance clears; for the
power system example, this is the system state at the moment when protection
action disconnects the lightning-affected transmission line.
The system will recover from the disturbance if and only if this IC lies in
the region of attraction (RoA) of the desired SEP.
As the parameter values of physical systems are uncertain and time-varying
in practice, it is particularly valuable to determine the set of parameter
values for which the system is able to recover to this SEP, which we call
the {\it recovery set} and denote by $R$.
We call a parameter value whose corresponding IC lies in the boundary of the
RoA of the corresponding desired SEP a {\it boundary parameter value},
because we will show (see Theorem~\ref{thm:multi}) that $\partial R$
often (in a precise sense defined in Section~\ref{sec:params})
consists entirely of boundary parameter values.
Prior algorithms were developed to determine or approximate $R$ by
numerically computing boundary parameter values \cite{Fi16,Fi18}.
The primary objective of this work is to provide a theoretical foundation
for those algorithms for a large class of parameter dependent vector fields.

Consider a particular boundary parameter value $p^*$.
Let a critical element refer to either an equilibrium point or a periodic orbit.
Suppose the orbit
of the IC corresponding to $p^*$ converges to a critical element
in the boundary of the RoA of the corresponding
desired SEP\@.
Then the amount of time the trajectory corresponding to $p^*$ spends in
any neighborhood of this critical element is infinite.
By continuity of the flow, it seems reasonable to expect that as parameter
values approach $p^*$, the time that the corresponding trajectory spends in this neighborhood
diverges to infinity.
Hence, to compute boundary parameter values, the algorithms begin by
identifying a special critical element in the boundary of the RoA,
which we call the {\it controlling critical element},
place a ball of fixed radius in state space around that controlling
critical element, and vary parameter values so as to maximize the
time that the system trajectory spends inside this ball.
As the time in the ball increases, the goal is that the parameter value will be
driven towards $\partial R$
(preferably the point in $\partial R$ that is closest to the original
parameter value). Building on this idea, algorithms have been developed to numerically estimate $R$ either by tracing $\partial R$ directly for the case of two dimensional parameter space, or by finding the largest ball around an
initial parameter value in $R$ that does not intersect $\partial R$.

The main technical challenge behind the theoretical justification
of these algorithms is to show that the RoA boundary
varies continuously in an appropriate sense under small changes in
parameter values (see Corollaries~\ref{cor:euc}-\ref{cor:morse}).
To illustrate this point, Example~\ref{ex:haus} shows that when the boundary of
the RoA does not vary continuously about a particular parameter value, then it is
possible for the ICs to ``jump'' over the RoA boundary.
In this case, $\partial R$ may not consist of boundary parameter values
and, in fact, its possible that boundary parameter values may not even exist.
The former implies that computation of boundary parameter values may not
provide an accurate estimate of $\partial R$ and hence of $R$,
and the latter implies that any attempt to compute boundary parameters must
fail since they don't exist;
both are problematic for the algorithms \cite{Fi16,Fi18} mentioned above.
Furthermore, discontinuity of the RoA boundary implies that there may not
exist a controlling critical element in the RoA
boundary with the property that as parameter values approach boundary values,
the time the trajectory spends in a ball around the controlling critical
element diverges to infinity.
Hence, even if boundary parameter values exist, the strategy employed by the
algorithms in \cite{Fi16,Fi18} may be unable to compute them.
Therefore, establishing continuity of the RoA boundary for a large class
of parameter dependent vector fields is crucial for motivating such algorithms.

The stable manifold of a critical element is the set
of initial conditions in state space which converge to that critical element
in forward time.
The approach that is used to establish continuity of the RoA boundary is to
show that at a fixed parameter value the RoA boundary is equal to the union
of the stable manifolds of the critical elements it contains, and that
this decomposition persists for small changes in parameter values, for a
large class of parameter dependent vector fields
(see Theorem~\ref{thm:bound2}).
Earlier work \cite{Ch88} reported this decomposition result for a large
class of fixed parameter $C^1$ vector fields on Euclidean space.
However, their proof relied on a Lemma \cite[Lemma~3-5]{Ch88}
which has been disproven \cite{Fi20}.
Therefore, we begin by providing a complete proof for a fixed parameter
RoA boundary decomposition result, and then focus on our main goal of
extending this work to a 
parametrized family of $C^1$ vector fields on either Euclidean space
or a compact Riemannian manifold.
Finally, continuity of the RoA boundary is then used to establish the
existence of a controlling critical element
possessing the properties that the time spent by the trajectory in a
neighborhood of the controlling critical element is continuous with respect to
parameter values and diverges to infinity as the parameter values
approach $\partial R$, thereby providing justification for the prior algorithms \cite{Fi16,Fi18}.

The paper is organized as follows.
Section~\ref{sec:defs} presents relevant background and notational
conventions.
Section~\ref{sec:examples} provides an example motivating 
discontinuity of the RoA boundary and the negative implications this can have.
Section~\ref{sec:results} presents the main results,
focusing on parameter dependent vector fields and controlling critical
elements, although results for parameter independent vector fields are
also included.
A simple example is provided to illustrate the main theorems.
Section~\ref{sec:thm11} proves the boundary decomposition results for the
case where the vector field is parameter independent.
Section~\ref{sec:thm21} builds on that foundation to
prove persistence of the boundary decomposition and continuity of the RoA
boundaries for a large class of parameter dependent vector fields. These boundary continuity results are applied in Section~\ref{sec:thm22} to
prove the existence of a controlling critical element with the
properties that motivate the algorithms of \cite{Fi16,Fi18}.
Finally, Section~\ref{sec:conc} offers some concluding thoughts and
future directions.

\section{Notation and Definitions}\label{sec:defs}



If $S$ is a subset of a topological space, we let $\overline{S}$ denote
its topological closure, $\partial S$ its topological boundary,
and $\text{int }S$ its topological interior.
If $f:A \to B$ is any function and $S \subset A$ is any subset of $A$,
$f|_S$ denotes the restricted function $f':S \to B$ defined by
$f'(s) = f(s)$ for all $s \in S$.
For a set $S$ contained in a metric space $K$,
define the $r$-neighborhood of $S$, denoted $S_r$, 
to be the set of $x \in K$
such that for each $x$ there exists $y \in S$ with $d(x,y) < r$.
Let $\{x_n\}_{n = 1}^\infty$ be a sequence.
Let $\{n_m\}_{m=1}^\infty$ be any collection of positive integers
where we require that $m' > m$ implies that $n_{m'} > n_m$ to ensure the
ordering is preserved.
Then any subsequence of $\{x_n\}_{n=1}^\infty$ can be written as
$\{x_{n_m}\}_{m=1}^\infty$ for some choice of $\{n_m\}_{m=1}^\infty$.
Let $K$ be a nonempty, compact metric space.
Note that compact Riemmanian manifolds are compact metric spaces, so they are
also covered by the following discussion.
Let $\mathcal{C}(K)$ be the nonempty, closed subsets of $K$.
Let $X,Y \in \mathcal{C}(K)$.
We define the Hausdorff distance $d_h$ by
\begin{align*}
d_h(X,Y) = \inf \{r \geq 0:X \subset Y_r, Y \subset X_r\}.
\end{align*}
Then $d_h$ is a well-defined metric on $\mathcal{C}(K)$
\cite[Section~28]{Ha57}
and we say a sequence of sets $A_n \in \mathcal{C}(K)$
converges to $A \in \mathcal{C}(K)$, denoted $A_n \to A$,
if $\lim_{n \to \infty} d_h(A_n,A) = 0$.
For $A, B$ subsets of a metric space with metric $d$, define a set distance
$d_S$ by $d_S(A,B) = \inf \{d(a,b):a \in A,b \in B\}$.
Then if $A$ is compact, $B$ is closed, and $d(A,B) = 0$,
$A$ and $B$ must have nonempty intersection.
As noted earlier, any Riemannian manifold is also a metric space, so this set distance
is well-defined on Riemannian manifolds.

Let $J$ be a topological space. For $p \in J$, we say that $p$ has a countable neighborhood basis if there exists a countable collection $\{U_n\}_{n=1}^\infty$ of open sets in $J$ that contain $p$ such that for any open set $U \subset J$ which contains $p$, there
exists an $n$ such that $U_n \subset U$. Then we say that $J$ is first countable if every point $p \in J$ possesses
a countable neighborhood basis.
If $J$ is first countable, $S$ is any topological space, and
$f:J \to S$ is a function, then $f$ is continuous if and only if for every
convergent sequence $\{p_n\}_{n=1}^\infty$ in $J$,
say $p_n \to p$, $f(p_n) \to f(p)$.
Let $J$ be a first countable topological space and let
$F:J \to \mathcal{C}(K)$.
We say that the family $\{A_p\}_{p \in J}$ is a Hausdorff continuous
family of subsets of $K$ if there exists $F:J \to \mathcal{C}(K)$
such that $F(p) = A_p$ for $p \in J$ and $F$ is continuous.
Since $J$ is first countable, $F$ is continuous if and only if for
every $p \in J$ and every sequence $p_n \in J$ with $p_n \to p$,
$F(p_n) \to F(p)$.


We consider another notion of convergence on $\mathcal{C}(K)$.
Let $A_n \in \mathcal{C}(K)$ be a sequence of sets.
Define $\liminf_{n \to \infty} A_n$ to be the set of points
$x \in K$ such that there exists a sequence $\{a_n\}$, with
$a_n \in A_n$ for all $n$, such that $a_n \to x$.
Define $\limsup_{n \to \infty} A_n$ to be the set of points
$x \in K$ such that there exist $\{a_{n_m}\}$ with
$a_{n_m} \in A_{n_m}$ a subsequence of $\{A_n\}$, such that
$a_{n_m} \to x$.
Both $\limsup_{n \to \infty} A_n$ and $\liminf_{n \to \infty} A_n$ are closed 
\cite[Section~28]{Ha57}
and $\limsup_{n \to \infty} A_n$ is nonempty since $K$ is sequentially
compact,
so if $\liminf_{n \to \infty} A_n$
is nonempty then both are elements of $\mathcal{C}(K)$.
By definition, $\liminf_{n \to \infty} A_n \subset \limsup_{n \to \infty} A_n$.
If $\limsup_{n \to \infty} A_n \subset \liminf_{n \to \infty} A_n$ then we
say the limit exists and 
$\lim_{n \to \infty} A_n = \limsup_{n \to \infty} A_n = \liminf_{n \to \infty} A_n$.
By statement~V of \cite[Section~28]{Ha57},
since $K$ is compact, if 
$\limsup_{n \to \infty} A_n = \liminf_{n \to \infty} A_n =$ \linebreak $\lim_{n \to \infty} A_n =: A$ then $\lim_{n \to \infty} d_h(A_n,A) = 0$.
Thus, if there exists $F:J \to \mathcal{C}(K)$ such that for every
$p \in J$ and every $p_n \to p$, 
$\limsup_{n \to \infty} F(p_n) = \liminf_{n \to \infty} F(p_n) = F(p)$,
then $\lim_{n \to \infty} d_h(F(p_n),F(p)) = 0$ so $\{F(p)\}_{p \in J}$ is a Hausdorff
continuous family of subsets of $K$.

Let $M = \mathbb{R}^n$ and let $\mathcal{C}(M)$ be the closed, nonempty
subsets of $M$.
The standard Hausdorff distance is not well-defined for unbounded
sets, so instead consider the one-point compactification of $M$, $\mathbb{R}^n \cup \{\infty\} \cong S^n$,
where $S^n$ is the $n$-sphere.
Equip $S^n$ with the induced Riemannian metric from its inclusion into
$\mathbb{R}^{n+1}$, and let its associated distance function be the
desired metric on $M \cup \infty$.
Then, since $S^n$ is a compact, nonempty metric space, the Hausdorff
distance is well-defined for all closed, nonempty subsets of $S^n$.
Let $\overline{\mathcal{C}}(M) = \{A \cup \{\infty\}:A \in \mathcal{C}(M)\}$.
Then all sets in $\overline{\mathcal{C}}(M)$ are closed and nonempty,
so the Hausdorff distance is well-defined on $\overline{\mathcal{C}}(M)$,
and the metric topology it induces on $\overline{\mathcal{C}}(M)$
is called the Chabauty topology.
From the discussion above regarding Hausdorff continuity, it follows 
that if there exists $F:J \to \overline{\mathcal{C}}(M)$ such that
for every $p \in J$ and every $p_n \to p$, 
$\limsup_{n \to \infty} F(p_n) = \liminf_{n \to \infty} F(p_n) = F(p)$,
then $d_h(F(p_n),F(p)) \to 0$ so $\{F(p)\}_{p \in J}$ is a Chabauty
continuous family of subsets of $M$.


Let $M$ be a Riemannian manifold. For each $x \in M$, let $T_xM$ denote the tangent space to $M$ at $x$.
Then the tangent bundle is given by $TM = \sqcup_{x \in M} T_xM$, where $\sqcup$ denotes the disjoint union\footnote{If $\{D_x\}_{x \in S}$ is a family
of sets $D_x$ parametrized by $x \in S$ for some set $S$, then the disjoint union of the family is $\sqcup_{x \in S} D_x = \bigcup_{x \in S} (x,D_x)$.}. Note that $TM$
is naturally a manifold with dimension twice that of $M$. Let the zero section be the subspace of $TM$ consisting of the zero vector
from each tangent space $T_xM$ over $x \in M$ (note that it is naturally
diffeomorphic to $M$ itself).
Note that a function $f:M \to N$, where $M$ and $N$ are $C^1$ manifolds,
is a submersion if $df_y$ is surjective for every $y \in M$, where $df_y$ denotes the differential of $f$ at $y$.
Let $X \subset M$ and $i:X \to M$ be the inclusion map, so $i(x) = x$ for all $x \in X$. We say that $i$ is a $C^1$ immersion if it is $C^1$ and for every
$y \in X$, $di_y$ is injective.
Then we say $X$ is an immersed submanifold if $i$ is a $C^1$ immersion.
Let $T_XM = \sqcup_{x \in X}T_xM$ denote the tangent bundle of $M$ over $X$.
Consider a pair of $C^1$ immersed submanifolds $X$ and $Y$.
We say that $X$ and $Y$ are transverse at a point $x \in X \cap Y$
if $T_xX \oplus T_xY$ spans $T_xM$.
Then we say that $X$ and $Y$ are transverse if for every $x \in X \cap Y$,
$X$ and $Y$ are transverse at $x$.
Note that if $X$ and $Y$ are disjoint, they are vacuously transverse.
A $C^1$ disk is the image of $i:B \to M$ where $B \subset \mathbb{R}^m$
is a closed ball around the origin in some Euclidean space $\mathbb{R}^m$,
and $i$ is a $C^1$ immersion.
A continuous family of $C^1$ disks is a parametrized family
$\{D_x\}_{x \in S}$ where $S$ is a topological space, $D_x$ is a $C^1$ disk
for each $x$, and $\{D_x\}_{x \in S}$ is a Hausdorff continuous family.
Suppose that $A$ is a $C^1$ immersed submanifold of $B$, which is a
$C^1$ immersed submanifold of $C$.
By the tubular neighborhood theorem \cite[Theorem~6.24]{Lee13},
there exists a $C^1$ continuous family of pairwise disjoint disks
$\{D(x)\}_{x \in B}$ in $C$ centered along $B$ and transverse to it
such that their union is an open neighborhood of $B$ in $C$.
Taking the restriction $\{D(x)\}_{x \in A}$ gives a $C^1$ continuous family
of pairwise disjoint disks in $C$ centered along $A$ and transverse to $B$.
If $F:M \to N$ is a continuous and injective map between manifolds $M$
and $N$ of the same dimension, then by invariance of domain
\cite[Theorem~2B.3]{Ha01}, $F$ is an open map, which means the image
under $F$ of every open set is open.

Let $V$ be a $C^1$ vector field on a Riemannian manifold $M$.
An integral curve $\gamma$ of $V$ is a map from an open subset
$U \subset \mathbb{R}$ to $M$ such that for every $t \in U$,
$\d{}{t} \gamma(t) = V_{\gamma(t)}$.
A flow is a map $\phi:U \times M \to M$, where $U \subset \mathbb{R}$ is open,
such that for any $x \in M$, $\phi(\cdot,x)$ is an integral curve of $V$.
For any $C^1$ vector field $V$ on a Riemannian manifold $M$, there exists
a $C^1$ flow $\phi$ \cite[Theorem~9.12]{Lee13}.
We say that $V$ is complete if it possesses a flow $\phi$
defined on $\mathbb{R} \times M$.
For $T \in \mathbb{R}$, we let $\phi_T:M \to M$
by $\phi_T(x) = \phi(T,x)$.
Then $\phi_T$ is a $C^1$ diffeomorphism of $M$ for any $T$ since $\phi$ is
$C^1$ and $\phi_T^{-1} = \phi_{-T}$.

If $M$ and $N$ are Riemannian manifolds which are at least $C^1$,
Let $C^1(M,N)$ denote the set of $C^1$ maps from $M$ to $N$.
There are two common topologies that $C^1(M,N)$ can be equipped with:
the strong and weak $C^1$ topologies.
Full definitions of these are available in \cite[Chapter~2]{Hi76}, but
the properties of these topologies which are most important for this work
are summarized below.
The purpose of introducing these topologies is to provide a framework
for careful consideration of perturbations to vector fields, and to be able
to define a continuous family of vector fields in a suitable way.
We will typically equip $C^1(M,N)$ with the weak topology, denoted
$C^1_W(M,N)$.
A major benefit of the weak topology is that it has a complete metric,
which we denote $d_{C^1}$ and refer to as $C^1$ distance.
If $V$ is a $C^1$ vector field on $M$ then $V \in C^1(M,TM)$.
A pair of $C^1$ vector fields $V, W$ on $M$ are $\epsilon$ $C^1$-close
if $d_{C^1}(V,W) < \epsilon$,
where $d_{C^1}$ is the $C^1$ distance on $C^1(M,TM)$.
A (weak) $C^1$ perturbation to the vector field $V$ is a vector field $W$
such that $V, W$ are $\epsilon$ $C^1$-close for sufficiently 
small $\epsilon > 0$.
A parameterized family of $C^1$ vector fields $\{V_p\}_{p \in J}$ on $M$
is (weakly) $C^1$ continuous if the induced map $J \to C^1_W(M,TM)$
that sends $p$ to $V_p$ is continuous.
Note that this implies that $V:M \times J \to TM \times TJ$ defined
by $V(x,p) = (V_p(x),0)$ is a $C^1$ vector field on $M \times J$.
In the case of $M$ compact, the strong and weak topologies
on $C^1(M,N)$ coincide.
For $M$ noncompact, a strong $C^1$ perturbation to a vector field only involves
changes to that vector field on a compact set, whereas a weak $C^1$ perturbation
to that vector field could have changes that are unbounded. For example, \cite[Example 19-1, p. 359]{Ch15} shows that under
weak $C^1$ perturbations to the vector field, a new equilibrium point can appear
arbitrarily close to infinity (the equilibrium point
``comes in from infinity'').
This is not possible under strong $C^1$ perturbations to the
vector field because for any strong $C^1$ perturbation there exists
an open neighborhood of infinity on which the vector field remains unchanged
by the perturbation.
Hence, weak continuity of vector fields is a weaker assumption than
strong continuity.


There is a notion of a generic $C^1$ vector field, which is meant to represent
typical behavior, similar to the idea of probability one in a
probability space.
If a property holds for a generic class of $C^1$ vector fields, it is therefore
considered to be typical or usual behavior.
As there exist many pathological $C^1$ vector fields,
it is often advantageous to
restrict attention to certain classes of generic $C^1$ vector fields when
possible, and to prove results for generic vector fields that often would
not hold for arbitrary vector fields.
We follow this approach here.
In a topological space $M$, a Baire set \cite[Section~48]{Mu15}
is a countable intersection
of open, dense subsets of $M$.
A topological space $M$ is metrizable if there exists a metric on $M$
whose metric topology corresponds with the original topology on $M$.
It is completely metrizable if the resultant metric space is complete.
The Baire category theorem states that if the topology on $M$ is
completely metrizable, then every Baire set in $M$ is dense.
By the discussion above, $C^1(M,N)$ is a complete metric space for
$M$ and $N$ under consideration here, so every Baire set will be dense.
Suppose $P$ is a property that may be possessed by elements of a
topological space $M$.
Then $P$ is called a generic property if the set of elements in $M$
which possess the property $P$ contains a Baire set in $M$.
So, a property of vector fields is generic (with respect to the weak topology)
if the subset of vector fields
in $C^1_W(M,TM)$ that possess this property contains a Baire set.

An equilibrium point $x_e \in M$ is a singularity of the vector field,
i.e.,~$V(x_e) = 0$.  A periodic orbit $X \subset M$ is an integral
curve of $V$ where there exists $T > 0$ such that each point of
$X$ is a fixed point of $\phi_T$.
For each point $x \in X$, there exists a codimension-one 
embedded submanifold $S$ transverse to the flow, called a cross section,
and a neighborhood $U$ of $x$ in $S$ such that the Poincar\'e first
return map $\tau:U \to S$ is well-defined and $C^1$
\cite[Page~281]{HiSm74}.
We call $X \subset M$ a critical element if it is either an
equilibrium point or a periodic orbit.

A set $S \subset M$ is forward invariant if $\phi_t(S) \subset S$
for all $t > 0$.
It is backward invariant if $\phi_t(S) \subset S$ for all $t < 0$,
and invariant if it is both forward and backward invariant.
Note that critical elements are invariant.

Let $x_e$ be an equilibrium point.
Then $x_e$ is hyperbolic if $d(\phi_1)_{x_e}$ is a hyperbolic linear
map, i.e.~if it has no eigenvalues of modulus one. 
It is stable if every eigenvalue of $d(\phi_1)_{x_e}$ has modulus
less than one.
If $X$ is a periodic orbit then let $x \in X$,
$S$ a cross section centered at~$x$,
$U$ a neighborhood of $x$ in $S$,
and $\tau:U \to S$ the $C^1$ first return map.
Then $X$ is hyperbolic if $d\tau_x$ is a hyperbolic linear map.

If $X \subset M$ is a hyperbolic critical element then it possesses
local stable and unstable manifolds \cite[Chapter 6]{Ka99},
$W^s_{\text{loc}}(X)$ and $W^u_{\text{loc}}(X)$, respectively,
such that $\phi_t(W^s_{\text{loc}}(X)) \subset W^s_{\text{loc}}(X)$
and $\phi_{-t}(W^u_{\text{loc}}(X)) \subset W^u_{\text{loc}}(X)$ for all $t > 0$.
Furthermore, the local stable and unstable manifolds are chosen to be compact.
The stable and unstable manifolds of $X$ are then defined as
$W^s(X) = \bigcup_{t \leq 0} \phi_t(W^s_{\text{loc}}(X))$
and
$W^u(X) = \bigcup_{t \geq 0} \phi_t(W^u_{\text{loc}}(X))$, respectively.
By \cite[Chapter 6]{Ka99}, $W^s(X)$ consists of the set of $x \in M$
such that the forward time orbit of $x$ converges to $X$, and $W^u(X)$
consists of the set of $x \in M$ such that the backward time orbit of $x$
converges to $X$.
Note that they are invariant under the flow.
If $X$ is a hyperbolic periodic orbit and $S$ is a cross section of $X$
with $C^1$ first return map $\tau$, it is often convenient
to consider $W^s_{\text{loc}}(X) \cap S$ and $W^u_{\text{loc}}(X) \cap S$.
Then $\tau(W^s_{\text{loc}}(X) \cap S) \subset W^s_{\text{loc}}(X) \cap S$
and $\tau^{-1}(W^u_{\text{loc}}(X) \cap S) \subset W^u_{\text{loc}}(X) \cap S$.
Hence, we abuse notation and let $W^s_{\text{loc}}(X)$ refer either to
$W^s_{\text{loc}}(X)$ as defined above or to $W^s_{\text{loc}}(X) \cap S$ for some
cross section $S$ of $X$.
The distinction should be clear from context.
Define the notation $W^u_{\text{loc}}(X)$ analogously.

Let $X$ be a hyperbolic critical element with
$(W^s(X)-X) \cap (W^u(X)-X) \neq \emptyset$.
Then the orbit of each $x \in (W^s(X)-X) \cap (W^u(X)-X)$ is called a
homoclinic orbit.
If, in addition, 
$W^s(X)$ and $W^u(X)$ have nonempty, transversal intersection, 
then the orbit of each $x \in (W^s(X)-X) \cap (W^u(X)-X)$ is called a
transverse homoclinic orbit. 
Let $X$, $Y$ be hyperbolic critical elements with 
$(W^s(X)-X) \cap (W^u(Y)-Y) \neq \emptyset$.
Then the orbit of each $x \in (W^s(X)-X) \cap (W^u(Y)-Y)$ 
is called a heteroclinic orbit, 
and it is called a transverse heteroclinic orbit if the
intersection is transverse.
Let $X^1,...,X^n$ be a finite set of hyperbolic critical elements
with $X^n = X^1$.
If $(W^s(X^i)-X^i) \cap (W^u(X^{i+1})-X^{i+1})$ 
is nonempty and transverse for each
$i \in \{1,...,n-1\}$, then we call $\{X^i\}_{i=1}^n$ a
heteroclinic cycle.
If $X^1,X^2, ...$ is a sequence of hyperbolic critical elements with
$(W^s(X^i)-X^i) \cap (W^u(X^{i+1})-X^{i+1})$ 
nonempty and transverse for all $i$,
then we call $\{X^i\}_{i=1}^\infty$ a heteroclinic sequence.

Let $X$ be a hyperbolic critical element.
If $X$ is an equilibrium point, let $B = W^u_{\text{loc}}(X)$,
let $D$ be a $C^1$ disk in $M$ such that $D$ has nonempty,
transversal intersection with $W^s(X)$,
and let $f = \phi_1$ be the time-one flow.
If $X$ is a periodic orbit, let $B = W^u_{\text{loc}}(X) \cap S$,
where $S$ is a cross section of $X$, let $D$ be a $C^1$ disk
in $S$ such that $D$ has nonempty, transversal intersection
in $S$ with $W^s(X) \cap S$, and let $f$ be the $C^1$ first
return map defined on an open subset of $S$.
Suppose $\text{dim }D \geq \text{dim }B$, and let $q \in D \cap W^s(X)$.
Let $f^n = f \circ f \circ ... \circ f$ denote composition of $f$ with itself
a total of $n$ times.
Then the Inclination Lemma, otherwise known as the Lambda Lemma,
states \cite{Pa69} that for every $\epsilon > 0$ there exists $n_0 > 0$
such that $n \geq n_0$ implies a submanifold of
$f^n(D)$ containing $f^n(q)$ is $\epsilon~C^1$-close to $B$.
For convenience, we often omit the submanifold qualifier and implicitly
redefine (shrink) $D$ so that
$f^n(D)$ itself is $\epsilon~C^1$-close to $B$.

Let $V$ be a $C^1$ vector field on a Riemannian manifold $M$ with
corresponding flow $\phi$.
A point $x \in M$ is nonwandering for $V$ 
if for every open neighborhood $U$
of $x$ and every $T > 0$, there exists $t > T$ such that 
$\phi_t(U) \cap U \neq \emptyset$.
Let $\Omega(V)$ denote the set of nonwandering points for $V$ in $M$.
If $y \in M$, define its $\omega$-limit set to be
the set of points $x \in M$ such that there exists a sequence
$t_i \to \infty$ with $\phi_{t_i}(y) \to x$.
If $\gamma \subset M$ is an orbit, define its $\omega$-limit set to
be the $\omega$-limit set of any $y \in \gamma$, and note that this
is well-defined because all points on an orbit share the same
$\omega$-limit set.
Define the $\alpha$-limit set of an orbit analogously,
for $t_i \to -\infty$.
Write $\omega(\gamma)$ and $\alpha(\gamma)$ for the $\omega$-limit set
and $\alpha$-limit set, respectively, of the orbit $\gamma$.
Then $V$ is a Morse-Smale vector field if it satisfies:
\begin{enumerate}
\item $\Omega(V)$ is a finite union of critical elements.
\item Every critical element is hyperbolic.
\item The stable and unstable manifolds of each individual and all pairs of
critical elements have transversal intersection.
\end{enumerate}
By \cite{Sm62,Ku63},
Assumptions 2 and 3 are generic, whereas Assumption 1 is not.
Note that Morse-Smale vector fields were defined for compact
Riemannian manifolds $M$ \cite{Sm60}. We will see in Section~\ref{sec:ind} that an additional assumption (Assumption~\ref{as:inf1}) is necessary for $M = \mathbb{R}^n$.

Let $J \subset \mathbb{R}$ be an open interval representing parameter
values and fix $p_0 \in J$.
Let $J_r = \{p \in J:|p-p_0| < r\}$
and $J_{\overline{r}} = \{p \in J:|p-p_0| \leq r\}$.
For $Q \subset J$, let $M_Q = M \times Q$ and let $M_p = M_{\{p\}}$.
Let $\{V_p\}_{p \in J}$ be a $C^1$ continuous family of vector fields on $M$.
Suppose $X_{p_0}$ is a hyperbolic critical element of $V_{p_0}$ for some
$p_0 \in J$.
Then for $J$ sufficiently small, $p \in J$ implies there exists
a unique hyperbolic critical element $X_p$ of $V_p$ which is $C^1$-close to
$X_{p_0}$ \cite[Chapter 16]{HiSm74}.
This defines the family $\{X_p\}_{p \in J}$ of a critical element of
the vector fields $\{V_p\}_{p \in J}$.
To avoid ambiguity,
we reserve the phrase ``family of a critical element'' to refer to the family
obtained from a single critical element as the parameter value $p$
varies over $J$.
In particular, this implies that for each fixed parameter value $p$, the family
of a critical element will possess exactly one critical element of $V_p$.
Throughout the paper, for a fixed parameter value $p \in J$, it will sometimes
be convenient to think of a critical element $X_p$ as being a subset of $M$,
and sometimes as a subset of $M \times J$.
Therefore, we abuse notation and let $X_p$ denote a critical element
of $V_p$, where sometimes we consider $X_p \subset M$
and sometimes we consider $X_p \subset M \times \{p\} \subset M \times J$.
The distinction should be clear from context.
For $Q \subset J$, we write $X_Q = \bigcup_{p \in Q} X_p \subset M \times J$.
We write $W^s(X_Q) = \sqcup_{p \in Q} W^s(X_p) \subset M \times J$,
and $W^u(X_Q) = \sqcup_{p \in Q} W^u(X_p) \subset M \times J$.


\section{Motivating Example}\label{sec:examples}

\begin{example}[Lack of Hausdorff Continuity of Boundaries]
\label{ex:haus}

We show that every smooth manifold $M$ possesses a family of
smooth vector fields which is continuous with respect to the strong 
$C^\infty$ topology and such that the vector fields have a family of stable
equilibria whose boundaries of their regions of attraction are not 
Hausdorff continuous.
As the strong $C^\infty$ topology is the most restrictive of the standard
$C^r$ topologies, this implies that even such a high degree of regularity
is not sufficient to prevent a lack of Hausdorff continuity of the
boundaries.
We define the family of vector fields such that they are supported within
a single chart, and then extend them trivially to the entire manifold $M$
by declaring them to be zero outside this chart.
So, it suffices to consider $M = \mathbb{R}^n$.
Let $h_{(a,b)}:[0,\infty) \to [0,1]$ be a smooth bump function with
$h_{(a,b)}^{-1}(1) = [0,a]$ and $h_{(a,b)}^{-1}(0) = [b,\infty)$.
Let $p \in \mathbb{R}$, $e_1$ denote the first standard basis vector,
and define, for $x \in \mathbb{R}^n$,
\begin{align*}
V_p(x) = -x h_{(0,1.5)}(|x|) + p h_{(2,3)}(|x|)e_1.
\end{align*}
Then $\{V_p\}_{p \in (-0.2,0.2)}$ is continuous with respect to the weak
$C^1$ topology.
In fact, \linebreak $\{V_p\}_{p \in (-0.2,0.2)}$ is also continuous with
respect to the more restrictive strong $C^\infty$ topology, so
$\{V_p\}_{p \in (-0.2,0.2)}$ varies as smoothly as might be desired.
Furthermore, for $p \in (-0.2,0.2)$, the vector field $V_p$ is unchanged
outside a fixed compact set.
Nevertheless, despite the smoothness of $\{V_p\}_{p \in (-0.2,0.2)}$
and the fact that variations are restricted to a fixed compact set,
this family of vector fields exhibits discontinuity in the boundaries of the
regions of attraction of a family of stable equilibria.

For each $p$, $V_p$ has a stable equilibrium point near the origin,
call them $\{X^s_p\}_{p \in (-0.2,0.2)}$.
The case of $n=1$ is illustrated in Fig.~\ref{fig:Vps}, which shows
the vector field $V_p$ for a few values of $p$.
For $p = 0.1$, the vector field is positive for $x \in (-3,X_{0.1}^s)$,
driving initial conditions in this range towards $X^s_{0.1}$,
and is negative for $x$ greater than $X^s_{0.1}$ but less than about 1.1,
driving these initial conditions towards $X^s_{0.1}$ as well.
So, $W^s(X^s_{0.1}) \approx (-3,1.1)$ consists of a line segment and
$\partial W^s(X^s_{0.1}) \approx \{-3,1.1\}$ consists of the two points on the
boundary of the line segment.
In fact, for any $p \in (0,0.2)$, $W^s(X^s_p)$ will be a line segment that
includes $(-3,0)$, and $\partial W^s(X^s_p)$ will consist of the two points
on its boundary, one of which is $\{-3\}$.
For $p = -0.1$, the vector field is negative for
$x \in (X_{-0.1}^s,3)$,
driving initial conditions in this range towards $X^s_{-0.1}$,
and is positive for $x$ less than $X^s_{-0.1}$ but greater than about -1.1,
driving these initial conditions towards $X^s_{-0.1}$ as well.
So, $W^s(X^s_{-0.1}) \approx (-1.1,3)$ consists of a line segment and
$\partial W^s(X^s_{-0.1}) \approx \{-1.1,3\}$ consists of the two points on the
boundary of the line segment.
In fact, for any $p \in (-0.2,0)$, $W^s(X^s_p)$ will be a line segment that
includes $(0,3)$, and $\partial W^s(X^s_p)$ will consist of the two points
on its boundary, one of which is $\{3\}$.
Now consider the case where $p = 0$.
By analogous reasoning to the above, based on the sign of the vector field,
$W^s(X^s_0) = (-1.5,1.5)$.
So, $\partial W^s(X^s_0) = \{-1.5,1.5\}$.
But, we saw that as $p$ approaches zero from above,
$\partial W^s(X^s_p)$ contains the point $\{-3\}$,
and as $p$ approaches zero from below, $\partial W^s(X^s_p)$ contains
the point $\{3\}$, neither of which are contained in
$\partial W^s(X^s_0) = \{-1.5,1.5\}$.
Hence, the family $\{\partial W^s(X^s_p)\}_{p \in (-0.2,0.2)}$ is Hausdorff
discontinuous at $p = 0$ from both above and below.

\begin{figure}
\centering
\includegraphics[width=0.5\textwidth]{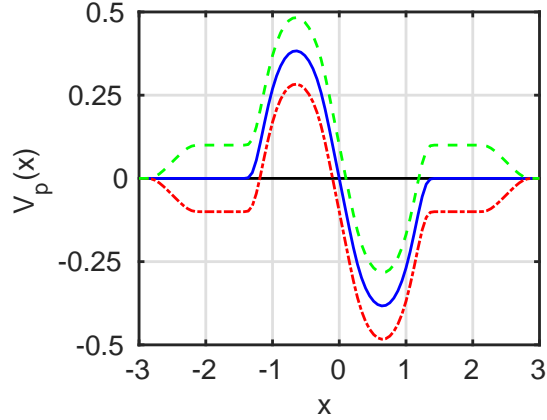}
\caption{The graph of $V_p$ for $p = 0.1$ (green dashed),
$p = 0$ (blue solid), and $p = -0.1$ (red dot-dashed).
This figure originally appeared in \cite{Fi17}.}
\label{fig:Vps}
\end{figure}

\begin{figure}
\centering
\includegraphics[width=0.5\textwidth]{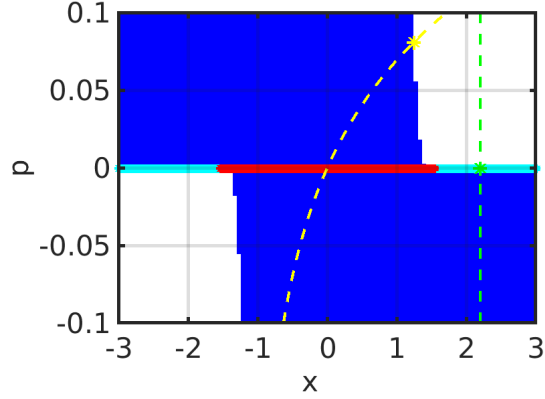}
\caption{The disjoint union of the regions of attraction of the family of 
stable equilibria of the vector fields $\{V_p\}$ over 
$p \in [-0.1,0.1]$ (blue).
The region of attraction of the stable equilibrium point of $V_0$
is shown in red.
Then $\partial W^s(X^s_0)$ consists of the two points on the boundary of this
red line segment, while
$\partial W^s(X^s_J) \cap \left(\mathbb{R} \times \{0\}\right)$
is equal to the union of the cyan line segments together with the end points of
the red line segment.
One family of initial conditions (yellow) begins inside the regions of
attraction and passes through one of their boundaries as $p$ is increased.
Another family of initial conditions (green) begins inside the regions of
attraction and passes outside without passing through one of their
boundaries.
This occurs because the boundaries of the regions of attraction fail to be
Hausdorff continuous at $p = 0$.
This figure originally appeared in \cite{Fi17}.}
\label{fig:bounds}
\end{figure}

Fig.~\ref{fig:bounds} illustrates this more clearly,
by showing $W^s(X^s_p)$ and $\partial W^s(X^s_p)$ for $p \in (-0.1,0.1)$,
as well as $\partial W^s(X^s_{(-0.1,0.1)})$.
Let $M = \mathbb{R}$ and let $J = (-0.1,0.1)$.
Then for $p \in J$ with $p \neq 0$, Fig.~\ref{fig:bounds} plots
$W^s(X^s_p) \subset M \times \{p\} \subset M \times J$ in blue.
And at $p = 0$,
$W^s(X^s_0) \subset M \times \{0\}$ is shown in red.
For $p > 0$, $W^s(X^s_p)$ includes $(-3,0)$, so $\partial W^s(X^s_p)$
contains $\{-3\}$.
For $p < 0$, $W^s(X^s_p)$ includes $(0,3)$, so $\partial W^s(X^s_p)$
contains $\{3\}$.
However, for $p = 0$, $W^s(X^s_0) = (-1.5,1.5)$,
so $\partial W^s(X^s_0) = \{-1.5,1.5\}$, which does not contain
$\{3\}$ nor $\{-3\}$.
Thus, as discussed above, $\{\partial W^s(X^s_p)\}_{p \in J}$
is Hausdorff discontinuous at $p = 0$ from both above and below.
Now consider $\partial W^s(X^s_J)$.
First note that $\partial W^s(X^s_J)$ contains
$\sqcup_{p \in J} \partial W^s(X^s_p)$, so for each $p \in J$ it contains the
two points of $\partial W^s(X^s_p)$.
However, as $\partial W^s(X^s_J)$ is obtained by taking the topological
boundary of $W^s(X^s_J)$ in $M \times J$, it also contains the cyan
line segments shown at $p = 0$, which are
$[-3,-1.5]$ and $[1.5,3]$.
Hence, $\partial W^s(X^s_J) \cap \left(M \times \{0\}\right)$
contains the line segments $[-3,-1.5]$ and $[1.5,3]$,
whereas $\partial W^s(X^s_0) = \{-1.5,1.5\}$ consists only of two points.
In particular, $\partial W^s(X^s_J)$ is strictly larger than
$\sqcup_{p \in J} \partial W^s(X^s_p)$.
For a large class of families of $C^1$ vector fields,
Theorem~\ref{thm:bound2} shows that
$\partial W^s(X^s_J) = \sqcup_{p \in J} \partial W^s(X^s_p)$,
and Corollaries~\ref{cor:euc}-\ref{cor:comp} show that
$\{\partial W^s(X^s_p)\}_{p \in J}$ varies Chabauty or Hausdorff continuously,
respectively.

From a practical perspective, we consider an initial condition $y_p$ which is
a $C^1$ function of parameter $p$ and represents the system state after a finite
time, parameter-dependent disturbance.
In order to prove Theorem~\ref{thm:time}, which provides theoretical motivation
for the prior algorithms of \cite{Fi16}, it is essential that there exists
a boundary parameter value $p^*$ such that $y_{p^*} \in \partial W^s(X^s_{p^*})$.
Suppose for some values of $p$ that $y_p \in W^s(X^s_J)$,
so the system recovers from the disturbance, and for
other values of $p$ that $y_p \not\in W^s(X^s_J)$,
so the system does not recover from the disturbance.
Then since $y_p$ is continuous in $p$, $y_J$ is connected, so there
must exist at least one $p^*$ such that $y_{p^*} \in \partial W^s(X^s_J)$.
However, as this example shows, $y_{p^*} \in \partial W^s(X^s_J)$ does not
necessarily imply that $y_{p^*} \in \partial W^s(X^s_{p^*})$ as is required
for the proof of Theorem~\ref{thm:time}.
In particular, Fig.~\ref{fig:bounds} shows two families of initial conditions
$y_J$: a yellow family of initial conditions which does pass through
$\partial W^s(X^s_{p^*})$ for some parameter value $p^*$, and a green family of
initial conditions which does not pass through $\partial W^s(X^s_p)$ for any
$p \in J$ but passes through $\partial W^s(X^s_J)$ via one of the cyan
line segments.
Hence, this example shows that 
when the assumptions required by Theorem~\ref{thm:bound2} are not
met, the conclusions of that theorem may not hold and,
as a result, the conclusions of Theorem~\ref{thm:time} may not hold either.

The discussion above generalizes to arbitrary dimension $n$. An example which shows that it is possible for a new nonwandering point to enter
the RoA boundary under arbitrarily small perturbations,
even if the vector field is globally Morse-Smale before the perturbation,
is given in \cite{Fi18b}.
In that example, a strong $C^1$
continuous family of Morse-Smale vector fields on $\mathbb{R}^2$
has a new equilibrium point enter the boundary of the RoA
for $p$ arbitrarily close to $p_0$, and the RoA boundary is Chabauty
discontinuous at $p_0$.
This motivates the need for Assumption~\ref{as:inf2} in
Section~\ref{sec:params}.

\end{example}

\section{Main Results}\label{sec:results}

\subsection{Vector Field is Parameter Independent}\label{sec:ind}

The primary motivation for presenting the results of this section
for parameter independent
vector fields is to provide a foundation for, and to improve the clarity of
presentation of, the results for parameter dependent vector fields
in Section~\ref{sec:params}.
However, the main result here (Theorem~\ref{thm:bound1})
may also be of some independent interest as
it provides a complete proof for parameter independent vector fields of
a result for which earlier proofs \cite{Ch88} are incomplete.

Let $V$ be a complete $C^1$ vector field on $M$,
where $M$ is either a compact Riemannian manifold or $\mathbb{R}^n$.
Let $X^s$ be a stable equilibrium point of $V$.
We make the following assumptions.

\begin{assumption}\label{as:wand1}
There exists a neighborhood $N$ of $\partial W^s(X^s)$ such that
$\Omega(V) \cap N$ consists of a finite union
of critical elements; call them $\{X^i\}_{i \in I}$ where $I = \{1,...,k\}$.
\end{assumption}

\begin{assumption}\label{as:inf1}
For every $x \in \partial W^s(X^s)$, the forward orbit of $x$ under
$V$ is bounded.
\end{assumption}

\begin{assumption}\label{as:hyp1}
Every critical element in $\partial W^s(X^s)$ is hyperbolic.
\end{assumption}

\begin{assumption}\label{as:trans1}
For each pair of critical elements in $\partial W^s(X^s)$,
say $X^i$ and $X^j$, 
$W^s(X^i)$ and $W^u(X^j)$ are transversal.
\end{assumption}

\begin{remark}
Assumptions~\ref{as:wand1},\ref{as:hyp1}, and \ref{as:trans1}
ensure that $V$ is Morse-Smale along \\$\partial W^s(X^s)$.
\end{remark}

\begin{remark}
Assumption~\ref{as:inf1} is necessary in the case $M = \mathbb{R}^n$
since Morse-Smale vector fields were defined
on compact manifolds \cite{Sm60},
whereas for $M = \mathbb{R}^n$ it is necessary to prohibit
orbits in $\partial W^s(X^s)$ from diverging to infinity in forward time.
\end{remark}

\begin{remark}
By the Kupka-Smale Theorem for $M$ compact \cite{Sm60,Ku63},
and its generalization for $M$ $\sigma$-compact \cite[Page~294]{Ka99},
Assumptions~\ref{as:hyp1} and \ref{as:trans1} are generic
with respect to the weak $C^1$ topology.
\end{remark}

\begin{remark}
\label{rem:wand}
By \cite[Remark 4.4]{Fi18b} and \cite[Lemma 4.5]{Fi18b},
Assumption~\ref{as:wand1} can be relaxed to
the assumption that there exists a neighborhood of $\partial W^s(X^s)$
in which the number of equilibrium points and periodic orbits is finite,
together with an additional assumption that is generic with respect to the
strong $C^1$ topology.
\end{remark}

\begin{remark}
By Assumption~\ref{as:hyp1}, hyperbolicity of the critical elements
implies that their stable and unstable manifolds exist.
\end{remark}

\begin{remark}
Assumptions~\ref{as:wand1}-\ref{as:inf1} together imply that for any
orbit $\gamma \subset \partial W^s(X^s)$,
$\omega(\gamma) = X^i$ for some $i \in I = \{1,...,k\}$.
\end{remark}



Theorem~\ref{thm:bound1} 
gives a decomposition of the boundary of the region 
of attraction for
a parameter independent vector field as a union of the stable manifolds
of the critical elements it contains.

\begin{theorem}\label{thm:bound1}
Let $M$ be either a compact Riemannian manifold or Euclidean space,
and suppose $V$ is a $C^1$ vector field on $M$ satisfying
Assumptions~\ref{as:wand1}-\ref{as:trans1}
Let $\{X^i\}_{i \in I}$ be the critical elements contained in
$\partial W^s(X^s)$.
Then $\partial W^s(X^s) = \bigcup_{i \in I} W^s(X^i)$.
\end{theorem}

\begin{remark}
Theorem~\ref{thm:bound1} was originally reported in
\cite[Theorem~4-2]{Ch88} under slightly more general assumptions.
Namely, our Assumption~\ref{as:wand1} was replaced by the assumption that
for every $x \in \partial W^s(X^s)$, the trajectory of $x$ converges to a
critical element in forwards time.
Hence, the number of critical elements in $\partial W^s(X^s)$ was not
assumed to be finite, and the set of $\omega$ limit points in
$\partial W^s(X^s)$, rather than the nonwandering set on a neighborhood of
$\partial W^s(X^s)$,
was assumed to consist solely of critical elements
(in general the nonwandering set may be larger than the closure of the set
of $\omega$ limit points).
The main purpose for presenting Theorem~\ref{thm:bound1} 
under these more restrictive assumptions is that its treatment more closely parallels
the results and proofs of Theorem~\ref{thm:bound2} for the case
of parameter dependent vector fields.
For example, a finite number of critical elements is necessary to ensure that
all critical elements persist under small perturbations to the vector field.
It should also be noted, though, that the proof of \cite[Theorem~4-2]{Ch88} relies on
\cite[Lemma~3-5]{Ch88}, which has been disproven~\cite{Fi20}. Therefore, the proof of \cite[Theorem~4-2]{Ch88} is incomplete,
so the proof of Theorem~\ref{thm:bound1} presented here represents the first
complete proof of this result.
\end{remark}

\subsection{Vector Field is Parameter Dependent}\label{sec:params}

Next we generalize the above results to the case where the vector field
is parameter dependent.
Let $J$ be a connected smooth manifold representing a family of
parameters, and let $\{V_p\}_{p \in J}$ be a weak $C^1$ continuous family
of complete $C^1$ vector fields on $M$.
Let $V$ be the complete $C^1$ vector field on $M \times J$
defined by $V(x,p) = (V_p(x),0) \subset T_xM \times T_pJ$.
Let $\phi$ be the $C^1$ flow of $V$, where $\phi(t,x,p)$ denotes the flow at time $t \in \mathbb{R}$ from initial condition $x \in M$ of the vector field $V_p$.
For fixed $t$, we often write
$\phi_t:M \times J \to M \times J$ by $\phi_t(x,p) = \phi(t,x,p)$
and note that $\phi_t$ is a $C^1$ diffeomorphism for each~$t$.

Let $\{X^s_p\}_{p \in J}$ be a $C^1$ continuous family of stable
equilibria of the vector fields $\{V_p\}_{p \in J}$.
Let $X^s_J = \bigcup_{p \in J} X^s_p$
and let $W^s(X^s_J) = \sqcup_{p \in J} W^s(X^s_p)$.
In this setting, there are two different boundaries of 
regions of attraction to consider. First, for any fixed parameter value $p \in J$ we have
$\partial W^s(X^s_{p})$, where the topological boundary operation is taken in $M$. Second, we have $\partial W^s(X^s_J)$, where the topological boundary operation is taken in 
$M \times J$. It is always true that $\sqcup_{p \in J} \partial W^s(X^s_p) \subset \partial W^s(X^s_J)$, but the two boundaries may differ as in Example~\ref{ex:haus}. Therefore, we make assumptions regarding the behavior of $V$ along $\partial W^s(X^s_J)$ rather than along $\sqcup_{p \in J} \partial W^s(X^s_p)$.
For some fixed $p_0 \in J$ we make the following assumptions.

\begin{assumption}\label{as:wand2}
There exists a neighborhood $N$ of $\partial W^s(X^s_J) \cap M_{p_0}$ in
$M_{p_0}$ such that $\Omega(V_{p_0}) \cap N$ consists of a finite union of critical elements of $V_{p_0}$;
call them $\{X^i_{p_0}\}_{i \in I}$ where $I = \{1,...,k\}$ and $k \geq 1$.
\end{assumption}

\begin{assumption}\label{as:hyp2}
Every critical element in $\partial W^s(X^s_J) \cap M_{p_0}$ is hyperbolic in $M$ with respect to $V_{p_0}$.
\end{assumption}

\begin{remark}\label{rem:hyp}
By Assumption~\ref{as:hyp2}, the critical elements $\{X^i_{p_0}\}_{i \in I}$ in
$\partial W^s(X^s_J) \cap M_{p_0}$ are hyperbolic
so, since $I$ is finite, 
they and their stable and unstable manifolds persist for
$J$ sufficiently small.
Let $X^i_p$ denote the perturbation of $X^i_{p_0}$ for $i \in I$ and $p \in J$.
Let $W^s(X^i_p)$ and $W^u(X^i_p)$ denote the stable and unstable manifolds,
respectively, for each $i \in I$ and $p \in J$.
\end{remark}

\begin{assumption}\label{as:inf2}
For each $p \in J$, $\Omega(V_p) \cap \left(\partial W^s(X^s_J) \cap M_p\right)
= \bigcup_{i \in I} X^i_p$ and for every $x \in \partial W^s(X^s_J) \cap M_p$
its forward orbit under $V_p$ is bounded.
\end{assumption}

\begin{assumption}\label{as:trans2}
For each pair of critical elements that are contained in
$\partial W^s(X^s_J) \cap M_{p_0}$, say $X^i_{p_0}$ and $X^j_{p_0}$,
$W^s(X^i_{p_0})$ and $W^u(X^j_{p_0})$ are transversal in $M$.
\end{assumption}

\begin{remark}
Assumptions~\ref{as:wand2}, \ref{as:hyp2} and \ref{as:trans2} 
are straightforward generalizations
of Assumptions~\ref{as:wand1}, \ref{as:hyp1} and \ref{as:trans1}.
They ensure that $V_{p_0}$ is Morse-Smale along $\partial W^s(X^s_J) \cap M_{p_0}$.
\end{remark}

\begin{remark}
Assumption~\ref{as:inf2} generalizes Assumption~\ref{as:inf1} 
by ensuring that, for every $p \in J$,
every orbit in $\partial W^s(X^s_J) \cap M_p$
converges to $X^i_p$ for some $i \in I$. This implies that the set $I$ indexing the critical elements remains unchanged for all $p \in J$,
and therefore that no critical elements enter or exit $\partial W^s(X^s_J)$ for $p \in J$.
%
\end{remark}

\begin{remark}
Using the results of \cite{Fi18b}, Assumption~\ref{as:wand2} can be partially
relaxed as in Remark~\ref{rem:wand}.
\end{remark}

\begin{remark}
If $M$ is a compact Riemannian manifold, Assumption~\ref{as:inf2} is not necessary, according to \cite[Theorem 4.6]{Fi18b}. If $M$ is Euclidean, \cite[Theorem 4.16]{Fi18b} allows Assumption~\ref{as:inf2} to be partially relaxed when $\{V_p\}_{p \in J}$ is a strong $C^1$ continuous family of vector fields. In particular, in this case it suffices to assume that for every $x \in \partial W^s(X^s_{p_0})$, the forward orbit of $x$ is bounded, that there exists a neighborhood $N$ of infinity such that $\Omega(V_{p_0}) \cap N = \emptyset$ and no orbit under $V_{p_0}$ is entirely contained in $N$ in both forward and backward time. There is also a requirement for some additional generic assumptions related to points of continuity of semi-continuous functions.
\end{remark}

Theorem~\ref{thm:bound2} gives a decomposition of $\partial W^s(X^s_J)$ as a disjoint union over parameter values in $J$ of a union of the stable manifolds of its critical elements. Furthermore, it shows that the topological boundary in $M \times J$, $\partial W^s(X^s_J)$, is equal to the disjoint union over $p \in J$ of the topological boundaries in $M$ of the stable manifolds of the stable
equilibria.
Using Theorem~\ref{thm:bound2}, it is straightforward to then show that
$\{\partial W^s(X^s_p)\}_{p \in J}$ is a continuous family of subsets
of $M$ (Corollary~\ref{cor:euc}).
Hence, if $M$ is a compact Riemannian manifold, this implies that
$\{\partial W^s(X^s_p)\}_{p \in J}$ is a Hausdorff continuous family of 
subsets of $M$ (Corollary~\ref{cor:comp}).
Finally, if $V_{p_0}$ is Morse-Smale on $M$ a compact Riemannian manifold,
using persistence of the so-called phase diagram of Morse-Smale vector
fields under perturbation \cite{Pa69}, one can show that for any
$C^1$ continuous family of vector fields $\{V_p \}_{p \in J}$ containing $V_{p_0}$,
and for $J$ sufficiently small, $\{\partial W^s(X^s_p)\}_{p \in J}$
is a Hausdorff continuous family of subsets of $M$ 
(Corollary~\ref{cor:morse}).  Analogous to $W^s(X^s_J)$, for each $i \in I$, let $W^s(X^i_J) = \sqcup_{p \in J} W^s(X^i_p)$.

\begin{theorem}\label{thm:bound2}
Let $M$ be either a compact Riemannian manifold or Euclidean space,
and let $\{V_p\}_{p \in J}$ be a family of vector
fields on $M$ continuous with respect to the weak $C^1$ topology
and satisfying Assumptions~\ref{as:wand2}-\ref{as:trans2}.
Let $\{X^i_{p_0}\}_{i \in I}$ denote the critical elements of $V_{p_0}$
in $\partial W^s(X^s_J) \cap M_{p_0}$.
Then in $M \times J$ for sufficiently small $J$,
$\partial W^s(X^s_J) = \sqcup_{p \in J} \partial W^s(X^s_p)
= \bigcup_{i \in I} W^s(X^i_J)$.
\end{theorem}

\begin{corollary}\label{cor:euc}
Let $M = \mathbb{R}^n$ and let $\{V_p\}_{p \in J}$ be a weak 
$C^1$ continuous family of vector
fields on $M$ satisfying Assumptions~\ref{as:wand2}-\ref{as:trans2}.
Then $\{\partial W^s(X^s_p)\}_{p \in J}$ is a Chabauty continuous
family of subsets of $M$.
\end{corollary}

\begin{corollary}\label{cor:comp}
Let $M$ be a compact Riemannian manifold
and let $\{V_p\}_{p \in J}$ be a $C^1$ continuous family of vector
fields on $M$ satisfying Assumptions~\ref{as:wand2}-\ref{as:trans2}.
Then $\{\partial W^s(X^s_p)\}_{p \in J}$ is a Hausdorff continuous
family of subsets of $M$.
\end{corollary}

\begin{corollary}\label{cor:morse}
Let $M$ be a compact Riemannian manifold and let $V_{p_0}$ be a
Morse-Smale vector field on $M$.  Then for any $C^1$ continuous
family of vector fields $\{V_p\}_{p \in J}$ on $M$ with $p_0 \in J$,
for sufficiently small $J$,
$\{\partial W^s(X^s_p)\}_{p \in J}$ is a Hausdorff continuous family
of subsets of $M$.
\end{corollary}

\subsection{Time in Neighborhood of Special Critical Element} \label{sec:time}

Recall from Section~\ref{sec:params} that $J$ is chosen to be a connected smooth manifold. Assume further, shrinking $J$ if necessary, that $\overline{J}$ is compact and convex.
Let $y:\overline{J} \to M$ send $p$ to the initial condition of $V_p$ and assume that
$y$ is $C^1$ over $\overline{J}$. We write $y_p := y(p)$ and, as with critical elements above, sometimes consider $y_p \in M$ and sometimes $y_p \in M \times \overline{J}$; the distinction should be clear from context. Then a parameter value $p^* \in \overline{J}$ is a boundary parameter value if and only if $y_{p^*} \in \partial W^s(X^s_{p^*})$. We restrict our attention to cases where $J$ contains points $p_1$ and $p_2$ such that $y_{p_1} \in W^s(X^s_{p_1})$ and $y_{p_2} \not\in W^s(X^s_{p_2})$. Let $R = \{p \in J:y_p \in W^s(X^s_p)\}$ and let $C = \{p \in \overline{J}:y_p \in \partial W^s(X^s_p)\}$. 
Then $R$ represents the set of parameters for which the system will recover
to the SEP, called the recovery set,
$C$ represents the set of boundary parameter values, and we let $\partial R$ denote the boundary of $R$ in $\overline{J}$.
Theorem~\ref{thm:multi} shows that $\partial R \subset C$ under the assumptions
of Section~\ref{sec:params}.
Furthermore, if $p_0 \in R$ is any parameter value in the recovery set and $J_0 \subset C$ is the set of parameter values which achieve minimum distance
from $p_0$ to $C$, then $J_0 \subset \partial R$ and
$J_0$ is the set of parameter values in $\partial R$ which achieve
minimum distance from $p_0$ to $\partial R$.


\begin{theorem}\label{thm:multi}
Let $M$ be either a compact Riemannian manifold or Euclidean space,
and let $\{V_p\}_{p \in J}$ be a family of vector
fields on $M$ continuous with respect to the weak $C^1$ topology
and satisfying Assumptions~\ref{as:wand2}-\ref{as:trans2}.
Let $y: \overline{J} \to M$ be $C^1$. Then $\partial R \subset C$. Fix any $p_0 \in R$ and let $J_0 = \{p^* \in C:d(p_0,p^*) = d_S(p_0,C)\}$.
Then $J_0$ is nonempty, $J_0 \subset \partial R$, and
$J_0 = \{p^* \in \partial R:d(p_0,p^*) = d_S(p_0,\partial R)\}$.
\end{theorem}

Theorem~\ref{thm:multi} justifies the method of determining or
approximating $R$ by computing the closest boundary parameter values.
Then Corollary~\ref{cor:multi} shows that for each boundary parameter
value $p^*$ there exists a critical element $X_{p^*}^*$, called the controlling
critical element, such that $y_{p^*}$ lies in its stable manifold.

\begin{corollary}\label{cor:multi}
Assume the conditions of Theorem~\ref{thm:multi}.
Fix any $p^* \in J_0$.
Then there exists a unique critical element $X^*_{p^*} \subset \partial W^s(X^s_J)$, 
called the {\it controlling critical element} corresponding to $p^*$,
such that $y_{p^*} \in W^s(X^*_{p^*})$.
\end{corollary}

For a fixed boundary parameter value $p^*$, by Corollary~\ref{cor:multi} there exists a unique controlling critical element $X^*_{p^*}$.  Since $X^*_{p^*} \subset W^s(X^s_J)$ is a critical element, by Assumption~\ref{as:inf2} there exists $j \in I$ such that $X^*_{p^*} = X^j_{p^*}$.  Furthermore, by Remark~\ref{rem:hyp}, as $X^j_{p^*}$ is hyperbolic, it persists over $p \in J$, and we can write $X^*_p = X^j_p$ for all $p \in J$. The notation $X^*_J = X^j_J$ is similarly defined.

Let $\gamma:[0,1] \to J$ be any $C^1$ path in $J$ such that
$\gamma([0,1)) \subset R$ and $\gamma(1) \in C$.
Let $X^*_{\gamma(1)}$ be the controlling critical element corresponding to $\gamma(1)$,
as in Corollary~\ref{cor:multi}.
Consider the following assumption regarding the path $\gamma$.

\begin{assumption}\label{as:neigh2}
Let $\gamma$ be a $C^1$ path in $J$ such that $\gamma([0,1)) \subset R$
and $\gamma(1) \in C$.
There exists a compact codimension-zero smooth embedded submanifold with
boundary $N$ in $M$ such that for $p \in \gamma([0,1])$,
$X^*_p$ is contained in the interior of $N$,
$X^s_p$ and $y_p$ are disjoint from $N$,
and the orbit of $y_p$ under $V_p$ has nonempty, transversal intersection with
$\partial N$.
\end{assumption}
  
\begin{remark}
Unlike Assumption~\ref{as:trans2}, the transversality
condition of Assumption~\ref{as:neigh2} can be easily checked directly 
by numerical simulation, and the neighborhood $N$ adjusted accordingly
if necessary.
In applications, $N$ is typically taken to be a closed ball
and its radius is adjusted to ensure the transversality condition
of Assumption~\ref{as:neigh2} holds. 
\end{remark}

\begin{remark}
Assumption~\ref{as:neigh2} also ensures that the initial conditions
and the stable equilibria do not intersect the neighborhood $N$,
and that the controlling critical element $X^*_J$ is contained in $N$.
\end{remark}

Let $\gamma$ and $N$ be as in Assumption~\ref{as:neigh2}.
Let $\tau_N:\gamma([0,1]) \to [0,\infty]$ be given by
$\tau_N(p) = \int_0^\infty \mathbb{1}_N(\phi(t,y_p,p)) dt$
where $\mathbb{1}_N$ is the indicator function of $N$, with $\mathbb{1}_N(x) = 1$ if $x \in N$ and $\mathbb{1}_N(x) = 0$ if
$x \not\in N$.
Therefore, $\tau_N(p)$ measures the length of time the orbit of $V_p$
with initial condition $y_p$ spends in $N$.
Theorem~\ref{thm:time} shows that $\tau_N$ is well-defined
and continuous over $\gamma([0,1])$.
Since $y_{\gamma(1)} \in W^s(X_{\gamma(1)}^*)$ and $X^*_{\gamma(1)} \subset N$,
it will follow that $\tau_N(p)$ diverges to infinity as $p$ approaches
$\gamma(1)$ along the path $\gamma$.

\begin{theorem}\label{thm:time}
Assume the conditions of Theorem~\ref{thm:multi}.
Fix any $p_0 \in J$, let
$J_0 = \{p^* \in \partial R:d(p_0,p^*) = d_S(p_0,\partial R)\}$,
and fix any $p^* \in J_0$.
By Corollary~\ref{cor:multi}, there exists a unique critical element
$X^*_{p^*} \subset \partial W^s(X^s_J)$ such that $y_{p^*} \in W^s(X^*_{p^*})$.
Let $\gamma:[0,1] \to J$ be a $C^1$ path satisfying Assumption~\ref{as:neigh2}
and such that $\gamma(0) = p_0$, $\gamma(1) = p^*$, and
$\gamma([0,1)) \subset R$.
Then $\tau_N:\gamma([0,1]) \to [0,\infty]$ is well-defined and
continuous. In particular,
$\lim_{s \to 1} \tau_N(\gamma(s)) = \tau_N(p^*) = \infty$.
\end{theorem}

\subsection{Illustrative Example}

\begin{example}[Illustration of Main Theorems]

To illustrate the results of Theorems~\ref{thm:bound1},
\ref{thm:bound2}, \ref{thm:time} and Corollary~\ref{cor:euc} we consider the
simple example of a damped, driven nonlinear pendulum with constant
driving force.
The dynamics are given by,
\begin{align}
\dot{x}_1 &= x_2 \label{sys1a} \\
\dot{x}_2 &= -c_1\sin(x_1) - c_2x_2 + c_3, \label{sys1b}
\end{align}
where $c_1, c_2, c_3 > 0$ are real parameters and 
$x = (x_1, x_2) \subset \mathbb{R}^2$.
Physically, $x_1$ represents the angle of the pendulum,
$x_2$ its angular velocity, $c_1$ the square of the  natural
frequency of the pendulum (under the small angle approximation), $c_2$ a damping coefficient due to air drag, and $c_3$ the constant driving torque.
Eqs.~\ref{sys1a}-\ref{sys1b} 
can also be interpreted as an electrical generator with
$(x_1,x_2)$ the angle and angular velocity of the turbine,
$c_1$ a constant determining the electrical torque supplied by
the generator, $c_2$ a damping coefficient
due to friction, and $c_3$ the constant driving mechanical torque.
For the demonstration below, we set 
$c = (c_1,c_2,c_3) = (2,0.5,1.5)$ and we restrict $x_1$ to a single interval
of length~$2\pi$ since $x_1$ is defined modulo~$2\pi$.
Although $c_3$ is initially given the fixed value of~$1.5$, we let
$p \equiv c_3$ and will subsequently treat it as a free parameter,
setting $p_0 = 1.5$.
At $p_0$, this system possesses one stable equilibrium point $X^s_{p_0}$,
at $(0.848,0)$, one unstable equilibrium point $X^1_{p_0}$, at $(2.294,0)$, 
and no other nonwandering elements.
Variation of the value of $p$ over a range $J$ that contains $p_0$
then generates a $C^1$ continuous family
of vector fields, as well as families of equilibria $\{X^s_p\}_{p \in J}$
and $\{X^1_p\}_{p \in J}$.

We establish an initial condition to Eqs.~\ref{sys1a}-\ref{sys1b} as the output of the related system,
\begin{align}
\dot{z}_1 &= z_2 \label{sys2a} \\
\dot{z}_2 &= - c_2z_2 + c_3,  \label{sys2b}
\end{align}
starting from the stable equilibrium point $X_p^s$ and running for time $c_4 = 0.8$~sec, which is the length of time the disturbance is active. Let $\phi_d$ denote the flow of Eqs.~\ref{sys2a}-\ref{sys2b} and let $J = (1.3,2)$. Then the initial condition of Eqs.~\ref{sys1a}-\ref{sys1b} is given by $y_p : J \to \mathbb{R}^2$ with $y_p = \phi_d(c_4,X^s_p,p)$. If Eqs.~\ref{sys1a}-\ref{sys1b} are interpreted as an electrical generator, then Eqs.~\ref{sys2a}-\ref{sys2b} represent a short circuit on the terminals of the generator so that it can no longer supply any electrical torque. This is modeled by setting $c_1=0$ in Eqs.~\ref{sys1a}-\ref{sys1b}, which then gives Eqs.~\ref{sys2a}-\ref{sys2b}.

Fig.~\ref{fig:fixed_bound} shows $\partial W^s(X^s_{p_0})$.
Note that the intersection of $\partial W^s(X^s_{p_0})$ with the nonwandering
set is $X^1$, every orbit $\gamma \subset \partial W^s(X^s_{p_0})$ has
$\omega(\gamma) = X^1_{p_0}$,
$X^1_{p_0}$ is hyperbolic, and the transversality assumption is
vacuously true since $X^1_{p_0}$ is the only critical element in 
$\partial W^s(X^s_{p_0})$.
Therefore, the system satisfies Assumptions~\ref{as:wand1}-\ref{as:trans1},
so by Theorem~\ref{thm:bound1}
we must have $\partial W^s(X^s_{p_0}) = W^s(X^1_{p_0})$, as can be seen 
in Fig.~\ref{fig:fixed_bound}.

\begin{figure}
\centering
\includegraphics[width=0.45\textwidth]{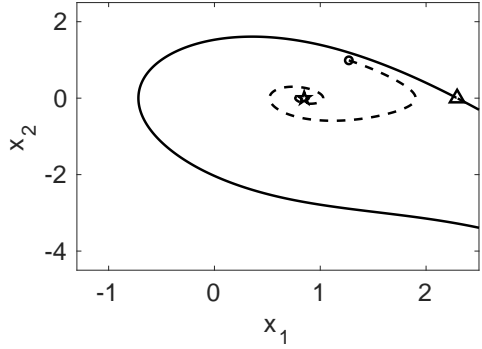}
\caption{The region of attraction boundary $\partial W^s(X^s_{p_0})$
  (solid black line) of the stable
  equilibrium point $X^s_{p_0}$ (black star) of Eqs.~\ref{sys1a}-\ref{sys1b} is
  shown.  It is equal to
  $W^s(X^1_{p_0})$ where $X^1_{p_0}$ (black triangle) is the unstable equilibrium
  point.  The orbit (dashed black line) from the initial condition
  $y_{p_0}$ (black circle) is shown.}
\label{fig:fixed_bound}
\end{figure}

\begin{figure}
\centering
\includegraphics[width=0.45\textwidth]{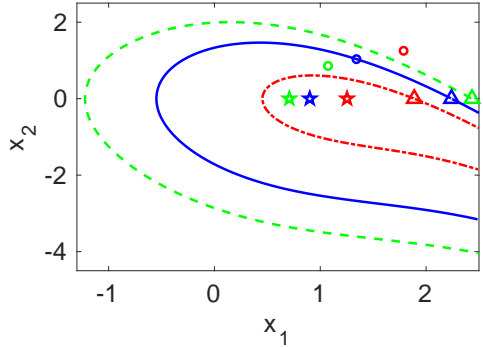}
\caption{The region of attraction boundaries $\partial W^s(X^s_p)$ of the
  stable equilibrium
  points $X^s_p$ (stars) for parameter values $p = 1.3$ (green dashed),
  $p = 1.568$ (solid blue), and $p = 1.9$ (red dot dashed) are shown.
  Each boundary is equal to $W^s(X^1_p)$ where $X^1_p$ (triangle) is the 
  unstable equilibrium point corresponding to parameter value $p$.
  The initial conditions (circles) are shown.}
\label{fig:fibers}
\end{figure}

Fig.~\ref{fig:fibers} 
shows the boundaries of the regions of attraction of the family of vector
fields for several values of the parameter $p \equiv c_3$.
At $p = 2$ the stable and unstable equilibria $X^s_p$ and $X^1_p$
collide in a saddle-node bifurcation and annihilate each other,
so we must restrict attention to sufficiently small $J = (1.3,2)$.
Fix $p_0 = 1.5$ as above.
Then the intersection of the nonwandering set with
$\partial W^s(X^s_J) \cap M_{p_0}$ is $X^1_{p_0}$,
for every orbit $\gamma \subset \partial W^s(X^s_J)$ we have
$\omega(\gamma) \subset X^1_J$, $X^1_{p_0}$ is hyperbolic,
and the transversality condition for $\partial W^s(X^s_J) \cap M_{p_0}$
is vacuously satisfied since the only critical element in
$\partial W^s(X^s_J) \cap M_{p_0}$ is $X^1_{p_0}$.
Therefore, the system satisfies Assumptions~\ref{as:wand2}-\ref{as:trans2},
so by Theorem~\ref{thm:bound2}
we must have 
$\partial W^s(X^s_J) = \sqcup_{p \in J} \partial W^s(X^s_p) = W^s(X^1_J)$,
and by Corollary~\ref{cor:euc} 
$\{\partial W^s(X^s_p)\}_{p \in J}$ is a Chabauty continuous family of subsets of $M$.

Choose two values of $p$, call them $p_1$ and $p_2$,
such that $y_{p_1} \in W^s(X^s_J)$ but $y_{p_2} \not\in W^s(X^s_J)$.
In particular, we may choose $p_1 = 1.3 ~(=p_0)$ and $p_2 = 1.9$. Then $y_{p_1} = (1.07,0.86) \in W^s(X^s_J)$ and $y_{p_2} = (1.79,1.25)\not\in W^s(X^s_J)$, as could be
verified, for example, by numerical integration.
Furthermore, since $\phi_d$ is $C^1$ then $y$ is also.

Hence, by Theorem~\ref{thm:multi} there must exist a boundary parameter value
$p^*$ such that $y_{p^*} \in \partial W^s(X^s_{p^*})$
and $d(p_0,p^*) = d_S(p_0,\partial R)$.
We will see that $p^* = 1.568$ is the desired boundary parameter value.
Since $\partial W^s(X^s_{p^*}) = W^s(X^1_{p^*})$,
this implies that $y_{p^*} \in W^s(X^1_{p^*})$, so $X^*_{p^*} = X^1_{p^*}$. 
Let $\gamma:[0,1] \to J$ by $\gamma(s) = (1-s)p_0 + sp^*$.
Then $\gamma$ is $C^1$,
$\gamma(0) = p_0$, $\gamma(1) = p^*$, and $\gamma(s) \notin \partial R$ for
$s \in [0,1)$ because $\gamma$ is a minimal geodesic
and $d(p_0,p^*) = d_S(p_0,C) = d_S(p_0,\partial R)$ by Theorem~\ref{thm:multi}.
As $\gamma([0,1))$ is connected, it does not intersect $\partial R$,
and $\gamma(0) \in R$, we must have $\gamma([0,1)) \subset R$.

Let $N$ be the closed ball centered at $X^*_{p_1} = X^1_{p_1}$ of radius 
$r = 1$ in $\mathbb{R}^2$.
Fig.~\ref{fig:ball_params} 
shows $X^*_p$ and the orbit of Eqs.~\ref{sys1a}-\ref{sys1b}
for a range of initial conditions $y_p$ for $p \in [p_1,p_2]$.
In particular, one can infer that each orbit has nonempty, transversal
intersection with $\partial N$ for $p \in [p_1,p_2]$.
Furthermore, $X^*_J \subset N$, and $y_{[p_1,p_2]}$ and $X^s_J$ are disjoint from $N$.
Therefore, the path $\gamma$ defined above satisfies Assumption~\ref{as:neigh2}
so by Theorem~\ref{thm:time}
we must have that the time $\tau_N$ spent by the orbit in the neighborhood $N$ is well-defined and continuous over $\gamma([0,1]) = [p_0,p^*]$.
Fig.~\ref{fig:time_params} 
illustrates the dependence of $\tau_N$ on $p \in \gamma([0,1]) = [p_0,p^*]$.
One observes that $\tau_N$ is continuous and 
that $\tau_N$ diverges to infinity as $p$ converges to a fixed value 
$p^*$.
For $p = p^*$ Fig.~\ref{fig:fibers} shows (solid blue) that
$y_{p^*} \in \partial W^s(X^s_J)$.
Furthermore, $p \in \gamma([0,1)) = [p_0,p^*)$ implies that $y_p \in W^s(X^s_J)$.
Although $\tau_N$ is monotonic in this example, this need not be true
in general.

\begin{figure}
\centering
\includegraphics[width=0.45\textwidth]{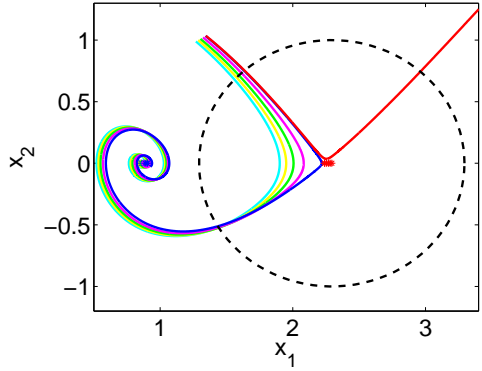}
\caption{The transversal intersection of several orbits with the ball 
$N$ containing the unstable equilibria (red stars).
Orbits are shown for parameter values
(driving torques) of 1.5 (cyan), 1.516 (yellow), 1.532 (green), 
1.55 (magenta), 1.568 (blue), and 1.57 (red). Only the initial
condition corresponding to the final parameter value of 1.57 lies outside
the region of attraction of the corresponding stable equilibrium.}
\label{fig:ball_params}
\end{figure}

\begin{figure}
\centering
\includegraphics[width=0.45\textwidth]{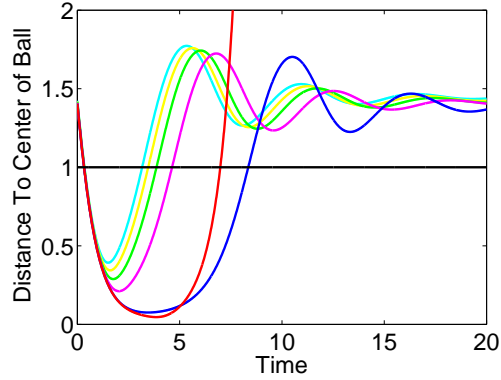}
\caption{Distance from the center of the ball $N$ as a function of
time for several orbits.  The line $r = 1$ marks the boundary of the ball
$\partial N$, so the time in the ball equals the difference in time between
the intersections of the orbit with this line.  The orbits shown correspond
to those in Fig.~\ref{fig:ball_params}.
As the parameter value approaches its boundary value from below, the
time in the neighborhood $N$ increases.  The final parameter value,
which is greater than the boundary parameter value, has an orbit 
(red) which spends less time in $N$ than that corresponding
to the boundary parameter value (blue).}
\label{fig:time_params}
\end{figure}


\end{example}

\section{Proof of Theorem~\ref{thm:bound1}}\label{sec:thm11}

This section is devoted to the proof of Theorem~\ref{thm:bound1}. Many of the results and proofs that underpin Theorem~\ref{thm:bound1} will be recycled for additional use for the parameter dependent vector field case in Section~\ref{sec:thm21}. Most of the lemmas presented here are similar to results given elsewhere, especially for diffeomorphisms of compact Riemannian
manifolds, but our presentation and proofs are novel unless
otherwise stated. In the following analysis, let $M$ be either a compact Riemannian manifold or Euclidean space unless stated otherwise.




Since $W^s(X^s)$ is invariant, its topological closure $\overline{W^s}(X^s)$ is invariant.
For, if $x \in \overline{W^s}(X^s)$ then there exists
a sequence $\{x_n\}_{n=1}^\infty \subset W^s(X^s)$ such that
$x_n \to x$.
By invariance of $W^s(X^s)$, $\phi_t(x_n) \in W^s(X^s)$ for all $n$ and $t \in \mathbb{R}$. By continuity of $\phi_t$, $\phi_t(x_n) \to \phi_t(x)$,
so $\phi_t(x) \in \overline{W^s}(X^s)$.
Hence, $\overline{W^s}(X^s)$ and $W^s(X^s)$ are invariant, so $\partial W^s(X^s) = \overline{W^s}(X^s) - W^s(X^s)$ is invariant.

Let $\{X^i\}_{i \in I}$ denote the critical elements in $\partial W^s(X^s)$.
Then $W^u_{\text{loc}}(X^i)$ and $W^s_{\text{loc}}(X^i)$
are well-defined local unstable and stable manifolds for $X^i$
for all $i \in I$.
Lemma~\ref{lem:tech1} provides a technical construction, for any
critical element, of a
compact set contained in its unstable manifold such that for any
sufficiently small neighborhood $N$ of this compact set in $M$, 
the following holds.
The union over time of
the time-$t$ flow $\phi_t$ of $N$ over all negative times $t$,
together with the stable manifold of the critical element, contains
an open neighborhood of the critical element in $M$.
This result will be instrumental in making the claim below that if
a critical element is contained in $\partial W^s(X^s)$ then its
unstable manifold intersects $\overline{W^s}(X^s)$.
Lemma~\ref{lem:tech1} is analogous to \cite[Corollary~1.2]{Pa69},
which states the corresponding result for diffeomorphisms without proof,
whereas here the result is shown for vector fields.
Fig.~\ref{fig:tech1} illustrates the content of Lemma~\ref{lem:tech1}.
Recall that if $D$ is a subset of a metric space and $\epsilon > 0$,
the notation $D_\epsilon$ refers to the subset of the metric space
such that for each $x \in D_\epsilon$ there exists $y \in D$ with
$d(x,y) < \epsilon$.

\begin{figure}
\centering
\includegraphics[width=0.45\textwidth]{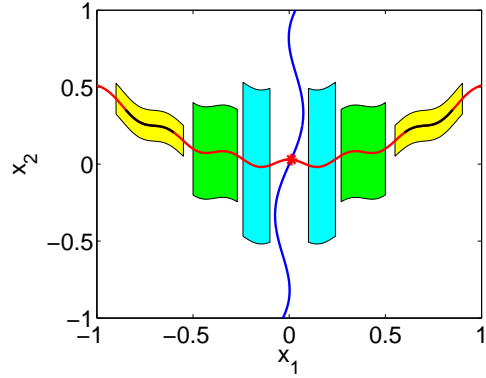}
\caption{The compact set $D$ (black line segments) and the neighborhood
$N$ (yellow shapes) mentioned in Lemma~\ref{lem:tech1} for an
equilibrium point (red star).
The set $D$ is contained in the unstable manifold (red line).
As the neighborhood $N$ is propagated backwards in time
(first to the green shapes then to the cyan), it approaches the
stable manifold (blue line) of the equilibrium point.
From the figure, it appears that the union of the backward flows of $N$
over all negative times, together with the stable manifold, will contain
a neighborhood of the equilibrium point, which is the content of
Lemma~\ref{lem:tech1}.
This figure originally appeared in \cite{Fi17}.}
\label{fig:tech1}
\end{figure}

\begin{lemma}\label{lem:tech1}
For any $i \in I$ and any $\epsilon > 0$ there exists a compact
set $D \subset W^u_{\text{loc}}(X^i) - X^i$ and an open neighborhood
$N$ of $D$ in $M$ disjoint from $X^i$ such that $N \subset D_\epsilon$ and
$\bigcup_{t \leq 0} \phi_t(N) \cup W^s(X^i)$
contains an open neighborhood of $X^i$ in $M$.
\end{lemma}

\begin{proof}[Proof Outline of Lemma~\ref{lem:tech1}]
If $X^i$ is an equilibrium point, let $f = \phi_1$ be the time-1 flow.
If $X^i$ is a periodic orbit, let $f = \tau$ be the first return map
of a cross section $S$ of $X^i$.
Let $D' = W^u_{\text{loc}}(X^i)$ and let $D$ be the topological closure of
$D' - f^{-1}(D')$ in $M$.
In order to show the existence of the desired open neighborhood of $X^i$,
the first step will be constructing a $C^1$ continuous disk family centered
along $D$ and contained in an open neighborhood $N \subset D_\epsilon$.
Then, this $C^1$ disk family is extended to a $C^1$ disk family
centered along
$W^u_{\text{loc}}(X^i)$ by backward iteration and the inclusion of the disk $W^s_{\text{loc}}(X^i)$. It is shown that this family is in fact $C^1$ continuous using the Inclination Lemma.
Finally, once the $C^1$ continuous disk family has been constructed, invariance of domain \cite[Theorem~2B.3]{Ha01} is applied
to conclude that the disk family contains an open neighborhood of $X^i$.
By construction, this implies that $\bigcup_{t \leq 0} \phi_t(N) \cup W^s(X^i)$ contains an open neighborhood of $X^i$. The full proof is provided in Appendix~\ref{ap:one}.
\end{proof}

We will use the technical result of Lemma~\ref{lem:tech1} to show
that the unstable manifold of a critical element in the boundary of the 
region of attraction must have nonempty intersection with
$\overline{W^s}(X^s)$.
The following lemma is analogous to the combination of
\cite[Theorem 3-3]{Ch88} (for equilibrium points in $\partial W^s(X^s)$) and
\cite[Corollary 3-4]{Ch88} (for periodic
orbits in $\partial W^s(X^s)$),
although \cite[Corollary 3-4]{Ch88} was unproven.
Our proof is similar to the proof of \cite[Theorem 3-3]{Ch88},
although we have explicitly proved Lemma~\ref{lem:tech1} whereas
\cite{Ch88} states a similar technical result without proof, and we also
explicitly prove \cite[Corollary 3-4]{Ch88}.

\begin{lemma}\label{lem:bound1}
If $X^i \subset \partial W^s(X^s)$
then $\{W^u(X^i) - X^i\} \cap \overline{W^s}(X^s) \neq \emptyset$.
\end{lemma}

\begin{proof}[Proof of Lemma~\ref{lem:bound1}]
Using Lemma~\ref{lem:tech1} we will produce a neighborhood of $X^i$
from its stable and unstable manifolds.
Since $X^i$ is in the topological boundary, this neighborhood must
intersect $W^s(X^s)$.
Then since stable manifolds cannot intersect, by invariance, and by
sending $\epsilon$ in the statement of Lemma~\ref{lem:tech1} to zero
we will obtain the result.

Let $\epsilon > 0$.
By Lemma~\ref{lem:tech1}, there exists a compact set 
$D \subset W^u_{\text{loc}}(X^i) - X^i$ and an open neighborhood
$N$ of $D$ in $M$ disjoint from $X^i$ such that $N \subset D_\epsilon$ and
$\bigcup_{t \leq 0} \phi_t(N) \cup W^s(X^i)$
contains a neighborhood of $X^i$ in $M$, call it $U_\epsilon$. Then $U_\epsilon$ is a neighborhood of $X^i \subset \partial W^s(X^s)$, so $U_\epsilon \cap W^s(X^s) \neq \emptyset$.
Since $W^s(X^i) \cap W^s(X^s) = \emptyset$, there must exist
some $T \leq 0$ such that $\phi_T(N) \cap W^s(X^s) \neq \emptyset$.
Since $W^s(X^s)$ is invariant, this implies that 
$N \cap W^s(X^s) \neq \emptyset$.
Since $N \subset D_\epsilon$, letting $d_S$ be the set distance
on the Riemannian manifold $M$, we have
\begin{align*}
\epsilon \geq d_S(D,W^s(X^s)) = d_S(D,\overline{W^s}(X^s))
\end{align*}
holds for all $\epsilon > 0$, so $d_S(D,\overline{W^s}(X^s)) = 0$.
Since $D$ is compact and $\overline{W^s}(X^s)$ is closed, this implies
that $D \cap \overline{W^s}(X^s) \neq \emptyset$.
Hence, since $D \subset W^u_{\text{loc}}(X^i) - X^i$, it must be that
$\{W^u(X^i) - X^i\} \cap \overline{W^s}(X^s) \neq \emptyset$.
\end{proof}

For $X$ a critical element,
let $n^t(X) = 0$ if $X$ is an equilibrium point and let
$n^t(X) = 1$ if $X$ is a periodic orbit.
Let $n^u(X) = \text{dim } W^u(X) - n^t(X)$
and let $n^s(X) = \text{dim }W^s(X) - n^t(X)$.
Lemma~\ref{lem:counting} was proven in \cite[Lemma~3.1]{Sm60}.
It is reproduced here for clarity of presentation.
A slightly different result, that was reported in \cite[Lemma 3-5]{Ch88}
and was fundamental in the proof of \cite[Theorem~4-2]{Ch88},
has been disproven \cite{Fi20}.

\begin{lemma}\label{lem:counting}
If $W^s(X^i) \cap W^u(X^j) \neq \emptyset$ then $n^u(X^i) \leq n^u(X^j) + n^t(X^j) - 1$, which is equivalent to $\text{dim }W^u(X^i) \leq \text{dim }W^u(X^j) + n^t(X^i) - 1$.
\end{lemma}

\begin{proof}[Proof of Lemma~\ref{lem:counting}]
Since $W^s(X^i)$ and $W^u(X^j)$ have a point of
transversal intersection and are invariant under the flow, they have
an orbit $\delta$ of transversal intersection.
Then 
$\dot{\delta}(0) \in T_{\delta(0)}W^s(X^i) \cap T_{\delta(0)}W^u(X^j)$. By transversality, $T_{\delta(0)}W^s(X^i) \oplus T_{\delta(0)}W^u(X^j) = T_{\delta(0)}M$. Let $L$ be the span of $\dot{\delta}(0)$ in $T_{\delta(0)}M$.
Then, since $L$ belongs to both of these tangent spaces, $\left(T_{\delta(0)}W^s(X^i)-L\right) \oplus \left(T_{\delta(0)}W^u(X^j)-L\right) = \left(T_{\delta(0)}M-L\right)$.
Thus, by dimensionality this implies that $(\text{dim }W^s(X^i)-1) + (\text{dim }W^u(X^j)-1) \geq n-1$. Hence,  $(n^s(X^i) + n^t(X^i)-1) + (n^u(X^j) + n^t(X^j)-1) \geq n-1$. Since $n^s(X^i) + n^t(X^i) + n^u(X^i) = n$, this implies that
$(n - n^u(X^i)-1) + (n^u(X^j) + n^t(X^j)-1) \geq n-1$
so $n^u(X^i) \leq n^u(X^j) + n^t(X^j) - 1$.
Hence,
$\text{dim }W^u(X^i) = n^u(X^i) + n^t(X^i) \leq \text{dim }W^u(X^j)
+ n^t(X^i) - 1$.
\end{proof}

As defined in Section~\ref{sec:defs}, a heteroclinic sequence is a sequence of hyperbolic critical elements such that the stable manifold of each critical element
intersects the unstable manifold of the next element of the sequence.
A heteroclinic cycle is a finite heteroclinic sequence where the first
and last critical elements are the same.
Lemmas~\ref{lem:trans}-\ref{lem:het} show that
Assumptions~\ref{as:wand1},\ref{as:hyp1},and \ref{as:trans1} 
imply that there are no heteroclinic cycles and, therefore, that all heteroclinic sequences are 
finite. These are analogous to several Lemmas in \cite{Pa69} for diffeomorphisms,
but are proved here for vector fields.
Lemma~\ref{lem:trans} shows that the intersection of stable and unstable 
manifolds of critical elements satisfies the transitive property.
It was shown in \cite[Corollary~1.3]{Pa69} for diffeomorphisms, and is
proven here for vector fields.

\begin{lemma}\label{lem:trans}
If $(W^s(X^i)-X^i) \cap (W^u(X^j)-X^j) \neq \emptyset$
and $(W^s(X^j)-X^j) \cap (W^u(X^k)-X^k) \neq \emptyset$
then $(W^s(X^i)-X^i) \cap (W^u(X^k)-X^k) \neq \emptyset$.
\end{lemma}

\begin{proof}[Proof of Lemma~\ref{lem:trans}]
The proof revolves around the openness of transversal intersection
of compact submanifolds which are $C^1$ close,
and the use of the Inclination Lemma to guarantee that the submanifolds
are $C^1$ close.

If $X^j$ is an equilibrium point, let $B = W^u_{\text{loc}}(X^j)$.
If $X^j$ is a periodic orbit, let $B = W^u_{\text{loc}}(X^j) \cap S$,
where $S$ is any cross section of $X^j$. By invariance of $W^s(X^i)$ and the assumptions of the Lemma, we have that $W^s(X^i) \cap B \neq \emptyset$. We claim that $B$ is transverse to $W^s(X^i)$. By Assumption~\ref{as:trans1}, $W^u_{\text{loc}}(X^j)$ is transverse to $W^s(X^i)$.
Hence, if $X^j$ is an equilibrium point then this implies that $B$ is transverse to $W^s(X^i)$.
Now suppose $X^j$ is a periodic orbit. For any $x \in W^s(X^i) \cap W^u(X^j)$,
$T_xW^s(X^i)$ and $T_xW^u(X^j)$ together span $T_x M$ since the intersection
is transverse. Then $B$ is obtained by intersecting $W^u_{\text{loc}}(X^j)$ with $S$,
so $T_xW^u(X^j)$ is equal to the span of $T_xB$ and the flow direction $V(x)$.
However, as $W^s(X^j)$ is invariant under $V$, $V(x) \in T_xW^s(X^j)$.
Therefore, $T_xW^s(X^i)$ and $T_xB$ together have the same span as
$T_xW^s(X^i)$ and $T_xW^u(X^j)$, which implies that $W^s(X^i)$ and $B$ are
transverse at $x$.  As $x$ was arbitrary, the claim follows.

By the definition of $W^s(X^i)$, there exists $T < 0$ such that 
$\phi_TW^s_{\text{loc}}(X^i) \cap W^u_{\text{loc}}(X^j) \neq \emptyset$.
Note that $B$ is a compact embedded submanifold, and that
it is transverse to $\phi_TW^s_{\text{loc}}(X^i)$ since it is transverse
to $W^s(X^i)$.
Since $\phi_TW^s_{\text{loc}}(X^i)$ and $B$ are compact submanifolds
with transversal intersection, by \cite[Corollary A.3.18]{Ka99}
there exists $\epsilon > 0$ such that if $D$ is a compact submanifold
which is $\epsilon$ $C^1$-close to $B$ then it has
nonempty, transversal intersection with $\phi_TW^s_{\text{loc}}(X^i)$,
and hence with $W^s(X^i)$.

Let $y \in (W^s(X^j)-X^j) \cap (W^u(X^k)-X^k)$.
Since by Assumption~\ref{as:trans1} the intersection is transverse, if $X^j$ is an equilibrium point
there exists a compact submanifold $D \subset W^u(X^k)$, which we choose to be a $C^1$ disk
centered at $y$ for the purpose of applying the Inclination Lemma, such that
$D$ is transverse to $W^s(X^j)$.
Similarly, if $X^j$ is a periodic orbit, then transversality of $W^u(X^k)$ and $W^s(X^j)$ in $M$ implies that $W^u(X^k) \cap S$ and $W^s(X^j) \cap S$ are transverse in $S$, so there exists
a $C^1$ disk $D \subset W^u(X^k) \cap S$ centered at $y$ such that $D$
is transverse to $W^s(X^j) \cap S$ in $S$.
By Lemma~\ref{lem:counting}, $\text{dim }W^u(X^k) \geq \text{dim }W^u(X^j)$,
so we may choose $D$ such that $\text{dim }D = \text{dim }B$.
Let $f = \phi_1$ if $X^j$ is an equilibrium point, and let
$f$ be a $C^1$ first return map on $S$ if $X^j$ is a periodic orbit.
Then, by the Inclination Lemma for equilibria or periodic orbits,
there exists $n_0 > 0$ such that 
$n \geq n_0$ implies $f^n(D)$ is $\epsilon$ $C^1$-close to $B$. By the choice of $\epsilon$, and the argument of the previous paragraph, $f^n(D) \cap W^s(X^i) \neq \emptyset$.
Since $D \subset W^u(X^k)$ invariant, this implies that
$W^s(X^i) \cap W^u(X^k) \neq \emptyset$.
\end{proof}

Lemma~\ref{lem:homo} shows that there are no homoclinic orbits
in $\partial W^s(X^s)$.
A similar claim was shown for diffeomorphisms in \cite[Corollary~1.4]{Pa69},
but the result here is proven for vector fields.

\begin{lemma}\label{lem:homo}
For any $X^i$, $W^s(X^i) \cap W^u(X^i) = X^i$.
\end{lemma}

\begin{proof}[Proof of Lemma~\ref{lem:homo}]
Using transversality and the Inclination Lemma we show that \linebreak
$W^s(X^i) \cap W^u(X^i)$ is nonwandering.
By Assumption~\ref{as:wand1}, this will imply that 
$W^s(X^i) \cap W^u(X^i) = X^i$.

Clearly $X^i \subset W^s(X^i) \cap W^u(X^i)$.
Assume towards a contradiction that
$\left(W^s(X^i)-X^i\right) \cap \linebreak \left(W^u(X^i)-X^i\right) \neq \emptyset$.
If $X^i$ is an equilibrium point, then by Lemma~\ref{lem:counting} this
implies that
$\text{dim } W^u(X^i) \leq \text{dim }W^u(X^i) - 1 < \text{dim }W^u(X^i)$,
which is a contradiction.
So, suppose $X^i$ is a periodic orbit, let $S$ be a $C^1$ cross section
of $X^i$, and let $B = W^u_{\text{loc}}(X^i) \cap S$.
By the assumption at the start of this paragraph and an invariance argument analogous to that in the proof of Lemma~\ref{lem:trans}, $B$ and $W^s(X^i)$ are transverse and $\left(W^s(X^i)-X^i\right) \cap \left(B-X^i\right) \neq \emptyset$. So, let $q \in \left(W^s(X^i)-X^i\right) \cap \left(B-X^i\right)$.
We claim that $q$ is nonwandering.
Let $U$ be any neighborhood of $q$ in $M$, and let $\epsilon > 0$ such that the ball of radius $\epsilon$ centered at $q$ is contained in $U$.
As $B$ is transverse to $W^s(X^i)$,
let $D \subset B$ be a $C^1$ disk centered at $q$ of the same dimension as
$B$ such that $D \subset U$ and $D$ is transverse to $W^s(X^i)$.
Note that $D$ is transverse to $W^s(X^i) \cap S$ in $S$ as well.
Let $f$ be a $C^1$ first return map on $S$.
Then, by the Inclination Lemma there exists $n_0 > 0$ such that
$n \geq n_0$ implies that $f^n(D)$ is $\epsilon$ $C^1$-close to $B$.
As $q \in B$ and $U$ contains the ball of radius $\epsilon$ centered
at $q$, this implies that $f^n(D) \cap U \neq \emptyset$. Hence, as $D \subset U$, we have that $f^n(U) \cap U \neq \emptyset$ for $n \geq n_0$, so $q$ is nonwandering.

By Assumption~\ref{as:wand1},
there exists a neighborhood $N$ of $\partial W^s(X^s)$ such that
$\Omega(V) \cap N = \bigcup_{j \in I} X^j$.
As $\omega(q) = \alpha(q) = X^i$ and $q \not\in X^i$,
$q \neq X^j$ for any $j \in I$.
But, since $X^i \subset N$ open and $\omega(q) = X^i$, there exists
$T > 0$ such that $\phi_T(q) \in N$.
As the nonwandering set is invariant, $\phi_t(q)$ is nonwandering in $N$.
As $X^j$ is invariant for each $j \in I$,
$\phi_t(q) \not\in \bigcup_{j \in I} X^j$, which is a contradiction
to the choice of $N$.
\end{proof}

Lemma~\ref{lem:het} now shows that every heteroclinic sequence has
finite length.

\begin{lemma}\label{lem:het}
There do not exist any heteroclinic cycles.
Hence, every heteroclinic sequence has finite length.
\end{lemma}

\begin{proof}[Proof of Lemma~\ref{lem:het}]
Assume towards a contradiction that $\{X^j\}_{j=1}^m$ is a heteroclinic
cycle.
By transitivity (Lemma~\ref{lem:trans}), since $X^m = X^1$, this implies that
$(W^s(X^1)-X^1) \cap (W^u(X^1)-X^1) \neq \emptyset$.
This contradicts Lemma~\ref{lem:homo}.

Since $\Omega(V) \cap \partial W^s(X^s)$ consists of a finite number
of critical elements, and since there are no heteroclinic cycles,
every heteroclinic sequence must be finite.
\end{proof}

Lemma~\ref{lem:open} will be used to complete the proof of 
Lemma~\ref{lem:int1}.
It is analogous to \cite[Lemma 7.1.b.]{Sm67}, but for vector fields
instead of diffeomorphisms.

\begin{lemma}\label{lem:open}
Suppose that $W^u(X^i) \cap W^s(X^s) \neq \emptyset$
and $W^s(X^i) \cap W^u(X^j) \neq \emptyset$.
Then $W^u(X^j) \cap W^s(X^s) \neq \emptyset$.
\end{lemma}

\begin{proof}[Proof of Lemma~\ref{lem:open}]
The proof uses the fact that $W^s(X^s)$ is open and, by invariance, 
intersects $W^u_{\text{loc}}(X^i)$, so any submanifold $K$ which is $C^1$ close
to $W^u_{\text{loc}}(X^i)$ also intersects $W^s(X^s)$.
The Inclination Lemma then guarantees that a disk in $W^u(X^j)$
is $C^1$ close to $W^u_{\text{loc}}(X^i)$.

Since $W^s(X^s)$ is invariant and intersects $W^u(X^i)$, 
$W^s(X^s) \cap W^u_{\text{loc}}(X^i) \neq \emptyset$.
So, let $q \in W^s(X^s) \cap W^u_{\text{loc}}(X^i)$.
By the definition of $W^s(X^s)$, there exists $T < 0$ such that
$q \in W^u_{\text{loc}}(X^i) \cap \phi_TW^s_{\text{loc}}(X^s)$.
If $X^i$ is an equilibrium point, let $B = W^u_{\text{loc}}(X^i)$.
If $X^i$ is a periodic orbit, let $S$ be a cross section containing
$q$ and let $B = W^u_{\text{loc}}(X^i) \cap S$.
Then it can be shown that $B$ is transverse to $\phi_TW^s_{\text{loc}}(X^s)$
(in $S$ if $X^i$ is a periodic orbit) by an argument analogous to that in the
proof of Lemma~\ref{lem:trans}.
Since $\phi_TW^s_{\text{loc}}(X^s)$ and $B$ are compact submanifolds
with transversal intersection,
by \cite[Proposition A.3.16,Corollary A.3.18]{Ka99}
there exists $\epsilon > 0$ such that if $D'$ is a compact submanifold
which is $\epsilon$ $C^1$-close to $B$ then it has a point of
transversal intersection with $\phi_TW^s_{\text{loc}}(X^s)$,
hence with $W^s(X^s)$.

Let $x \in W^s(X^i) \cap W^u(X^j)$.
Since the intersection is transversal by Assumption~\ref{as:trans1},
if $X^i$ is an equilibrium point 
there exists a $C^1$ disk $D \subset W^u(X^j)$ centered at $x$ with
$D$ transverse to $W^s(X^i)$.
Similarly, if $X^i$ is a periodic orbit there exists a disk
$D \subset W^u(X^j) \cap S$ centered at $x$ with $D$
transverse to $W^s(X^i) \cap S$ in $S$.
By Lemma~\ref{lem:counting}, $\text{dim }W^u(X^j) \geq \text{dim }W^u(X^i)$,
so we may choose $D$ such that $\text{dim }D = \text{dim }B$.

If $X^i$ is an equilibrium point let $f = \phi_1$, and if $X^i$ is a periodic
orbit let $f$ be a $C^1$ first return map for $S$.
Then, by the Inclination Lemma for equilibria or periodic orbits,
there exists $n_0 > 0$ such that 
$n \geq n_0$ implies $f^n(D)$ is $\epsilon$ $C^1$-close to $B$.
By the choice of $\epsilon$, $f^n(D) \cap W^s(X^s) \neq \emptyset$.
Since $D \subset W^u(X^j)$ invariant, this implies that $W^u(X^j) \cap \linebreak W^s(X^s) \neq \emptyset$.
\end{proof}

Lemma~\ref{lem:int1} was reported as \cite[Theorem~3-8]{Ch88},
where our Assumption~\ref{as:wand1} was replaced by the weaker assumption
that for every $x \in \partial W^s(X^s)$, the trajectory of $x$ converges to
a critical element in forward time.
However, the proof of \cite[Theorem~3-8]{Ch88} relies crucially
on \cite[Lemma~3-5]{Ch88}, which has been disproven \cite{Fi20},
to show that a particular heteroclinic sequence has finite length.
In contrast, the proof of Lemma~\ref{lem:int1} shows that
an analogous heteroclinic sequence has finite length. This result uses
Lemma~\ref{lem:het} which relies on Assumption~\ref{as:wand1}.

\begin{lemma}\label{lem:int1}
If $X^i \subset \partial W^s(X^s)$ then
$W^u(X^i) \cap W^s(X^s) \neq \emptyset$.
\end{lemma}

\begin{proof}[Proof of Lemma~\ref{lem:int1}]
We first construct a heteroclinic sequence of critical elements,
which must be finite by Lemma~\ref{lem:het}.
Then we show that the unstable manifold of the final critical element in 
the sequence intersects $W^s(X^s)$.
Working backwards, we argue that the unstable manifold of every
critical element in the sequence intersects $W^s(X^s)$
using Lemma~\ref{lem:open}, which implies the result.

The first step is the construction of the heteroclinic sequence
$\{X^j\}_{j \in \Lambda}$.
As $X^j \subset \partial W^s(X^s)$, by Lemma~\ref{lem:bound1} there exists 
$x_j \in (W^u(X^j)-X^j) \cap \overline{W^s}(X^s)$.
If $x_j \in W^s(X^s)$ then we have finished constructing the heteroclinic
sequence, so suppose $x_j \in \partial W^s(X^s)$.
Then by Assumptions~\ref{as:wand1}-\ref{as:inf1},
$x_j \in W^s(X^{j+1})$ for some critical element
$X^{j+1} \subset \partial W^s(X^s)$.
Iterating this procedure yields a heteroclinic sequence $\{X^j\}_{j \in \Lambda}$.
By Lemma~\ref{lem:het} it has finite length.
The final element of the sequence, call it $X^m$, must satisfy 
$W^u(X^m) \cap W^s(X^s) \neq \emptyset$, since otherwise
there would be another element $X^{m+1}$ that would be added to the
heteroclinic sequence by the procedure above.

We conclude by showing that the unstable manifold of each critical
element in the heteroclinic sequence must intersect $W^s(X^s)$,
which implies the result.
For any $j \in \Lambda$, suppose $W^u(X^j) \cap W^s(X^s) \neq \emptyset$.
By recursion, it suffices to show that this implies 
$W^u(X^{j-1}) \cap \linebreak W^s(X^s) \neq \emptyset$.
However, by the construction of the sequence we have that
$W^u(X^{j-1}) \cap \linebreak W^s(X^j) \neq \emptyset$ is a transversal
intersection.
Hence, the result follows from Lemma~\ref{lem:open}.
\end{proof}

\begin{proof}[Proof of Theorem~\ref{thm:bound1}]
Fix $i \in I$. By Lemma~\ref{lem:int1},
$W^u(X^i) \cap W^s(X^s) \neq \emptyset$.
To show that $W^s(X^i) \subset \partial W^s(X^s)$,
it suffices to show that $W^s_{\text{loc}}(X^i) \subset \partial W^s(X^s)$
since $\partial W^s(X^s)$ is invariant and by the definition of $W^s(X^i)$.
Let $x \in W^s_{\text{loc}}(X^i)$.
By the proof of Lemma~\ref{lem:open}, 
there exists a disk $D$ centered at $x$, contained in the
$\epsilon$-neighborhood
of $x$ in $M$, and transverse to $W^s_{\text{loc}}(X^i)$, such that
$\phi_t(D) \cap W^s(X^s) \neq \emptyset$ for some $t > 0$.
By invariance, $D \cap W^s(X^s) \neq \emptyset$.
Since $D$ is contained in the $\epsilon$-neighborhood of $x$ in $M$,
$d_S(x,\overline{W^s}(X^s)) = d_S(x,W^s(X^s)) \leq \epsilon$.
As this holds for all $\epsilon > 0$,
$d_S(x,\overline{W^s}(X^s)) = 0$.
Since $\{x\}$ is compact and $\overline{W^s}(X^s)$ is closed, this implies
that $x \in \overline{W^s}(X^s)$.
However, $x \in W^s_{\text{loc}}(X^i)$ implies that 
$x \in \partial W^s(X^s)$.
Thus, $W^s_{\text{loc}}(X^i) \subset \partial W^s(X^s)$,
so $W^s(X^i) \subset \partial W^s(X^s)$.
Hence $\bigcup_{i \in I} W^s(X^i) \subset \partial W^s(X^s)$.

By Assumption~\ref{as:inf1}, if $\gamma \subset \partial W^s(X^s)$ is an
orbit then $\omega(\gamma) = X^j$ for some $j \in I$,
which implies that $\gamma \subset W^s(X^j)$.
Thus, $\partial W^s(X^s) \subset \bigcup_{i \in I} W^s(X^i)$.
\end{proof}

\section{Proofs of Theorem~\ref{thm:bound2} and Corollaries}\label{sec:thm21}

The proofs of Theorem~\ref{thm:bound2} and its corollaries proceed by paralleling the treatment of the fixed parameter case in Section~\ref{sec:thm11}. The recurring strategy of the proofs of this section is to reduce to the fixed parameter case where possible, and then to rely on the results and proofs of Section~\ref{sec:thm11}
to complete the arguments.

Recall the notation from Section~\ref{sec:params}.  In particular, let $V$ denote the $C^1$ vector field on $M \times J$ defined by $V(x,p) = (V_p(x),0)$, let $\phi$ be the $C^1$ flow of $V$, and for any fixed $t \in \mathbb{R}$ let $\phi_t:M \times J \to M \times J$ be the $C^1$ diffeomorphism defined by $\phi_t(x,p) = \phi(t,x,p)$. For the remainder of this section, fix $p_0 \in J$ such that $p_0$ satisfies Assumptions~\ref{as:wand2}-\ref{as:trans2}.

We begin by defining functions whose images for each $p \in J$ are the
critical elements and their local stable and unstable manifolds for the
vector field $V_p$. As there are finitely many hyperbolic critical elements $\{X^i_{p_0}\}_{i \in I}$, we may
assume $J$ sufficiently small such that they and their local stable
and unstable manifolds are well defined and vary $C^1$ continuously
with parameter over $J$.
Let $S^s = W^s_{\text{loc}}(X^s_{p_0})$ and for $i \in I$, let $S^i = X^i(p_0)$,
$S^i_s = W^s_{\text{loc}}(X^i_{p_0})$, and $S^i_u = W^u_{\text{loc}}(X^i_{p_0})$.
As the critical elements and their local stable and unstable manifolds
vary $C^1$ continuously with parameter, there exist $C^1$ maps,
\begin{equation}
F^i:S^i \times J \to M, \quad F^i_s:S^i_s \times J \to M, \quad
F^i_u:S^i_u \times J \to M, \quad F^s:S^s \times J \to M, \label{eq:Fs}
\end{equation}
such that for any $p \in J$,
$F^i(\cdot,p)$, $F^i_s(\cdot,p)$, $F^i_u(\cdot,p)$, and $F^s(\cdot,p)$ are $C^1$
diffeomorphisms onto $X^i_p$, $W^s_{\text{loc}}(X^i_p)$,
$W^u_{\text{loc}}(X^i_p)$, and $W^s_{\text{loc}}(X^s_p)$, respectively. 
In other words, $F^i$, $F^i_s$, $F^i_u$, and $F^s$ describe quantitatively
how the critical elements and their local stable and unstable manifolds
vary $C^1$ with parameter $p$.
Let $\pi_J$ be the projection onto parameter space, $\pi_J(x,p) = p$.
The functions above have codomain $M$, but it will sometimes be convenient for the codomain to be $M \times J$.
To this end, let $G^i = (F^i,\pi_J)$, $G^i_s = (F^i_s,\pi_J)$,
$G^i_u = (F^i_u,\pi_J)$, and $G^s = (F^s,\pi_J)$, and note that these
functions are $C^1$ injections because for fixed $p \in J$ the functions
\eqref{eq:Fs} are $C^1$ diffeomorphisms onto their images.

Lemma~\ref{lem:open2} establishes properties of  $W^s(X^s_J)$ that will be used in subsequent developments.

\begin{lemma}\label{lem:open2}
$W^s(X^s_J)$ is open and invariant in $M \times J$.
\end{lemma}

\begin{proof}[Proof of Lemma~\ref{lem:open2}]
Since $S^s$ is equal to $W^s_{\text{loc}}(X^s_{p_0})$,
a codimension-zero embedded submanifold with boundary in $M$, $G^s|_{\text{int }S^s \times J}$ is a continuous injection between manifolds
of the same dimension so, by invariance of domain \cite[Theorem~2B.3]{Ha01},
an open map.
Thus, $G^s(\text{int }S^s \times J)$ is an open set in $M \times J$.
Hence, by definition of the local stable manifold,
$W^s(X^s_J) = \bigcup_{t \leq 0} \phi_t(G^s(\text{int }S^s \times J))$
is a union of open sets since $\phi_t$ is a $C^1$ diffeomorphism for
each $t$, hence open.
Since $W^s(X^s_J) = \sqcup_{p \in J} W^s(X^s_p)$ is a union of invariant sets, it is invariant.
\end{proof}

Let $p_0 \in J$ be a fixed parameter value such that
Assumptions~\ref{as:wand2}-\ref{as:trans2} hold.
Recall from Section~\ref{sec:defs} that the family of a critical element refers
here to the family obtained from a single critical element as the parameter
value is varied over $p \in J$.
Similar to Lemma~\ref{lem:tech1},
Lemma~\ref{lem:tech2} provides a technical construction, for
any critical element contained in $\partial W^s(X^s_J)$, of a compact set contained in its family of unstable manifolds.
The lemma proceeds to show that for any
sufficiently small neighborhood $N$ of this compact set in $M \times J$, 
the union over all negative times $t$ of the flow $\phi_t$ of $N$, together with the family of stable manifolds of the critical element, contains
an open neighborhood of the critical element in $M \times J$.
The key difference from the fixed parameter case Lemma~\ref{lem:tech1}
is that the open neighborhood that is contained
in the union is open in $M \times J$, whereas for Lemma~\ref{lem:tech1}
it was open in $M$ alone.
This is important because for a critical element contained in
$\partial W^s(X^s_J)$, an open neighborhood in $M \times J$ of that critical
element is required to guarantee it intersects $W^s(X^s_J)$.
This result will be fundamental in proving the claim that if
a critical element in $M_{p_0}$ is contained in $\partial W^s(X^s_J)$ then its
unstable manifold intersects $\overline{W^s}(X^s_J)$ in $M_{p_0}$.
Recall that if $D$ is a subset of a metric space and $\epsilon > 0$,
the notation $D_\epsilon$ refers to the subset of the metric space
such that for each $x \in D_\epsilon$ there exists $y \in D$ with
$d(x,y) < \epsilon$.

\begin{lemma}\label{lem:tech2}
For any $i \in I$ and any $\epsilon > 0$ sufficiently small, there exists a compact set $D \subset W^u_{\text{loc}}(X^i_J) - X^i_J$ and an open neighborhood $N$ of $D$ in $M \times J$ such that $N \subset D_\epsilon$, $D_\epsilon \cap X^i_J = \emptyset$, and $\bigcup_{t \leq 0} \phi_t(N) \cup W^s(X^i_J)$ contains an open neighborhood of $X^i_{p_0}$ in $M \times J$.
\end{lemma}

\begin{proof}[Proof Outline of Lemma~\ref{lem:tech2}]
If $X^i_{p_0}$ is an equilibrium point, let $f = \phi_1$ be the time-1 flow of the vector field $V$. If $X^i_{p_0}$ is a periodic orbit, let $f = \tau$ be the first
return map of a Poincar\'e cross section $S$. (Note that this map is
well-defined and $C^1$ with respect to parameter value $p \in J$.)
Let $D'_p = G^i_u(S^i_u \times\{p\})$ for any $p \in J$.
Let $D_p$ be the topological closure of $D'_p - f^{-1}(D'_p)$ in $M$.
We will prove the following claim: there exists an open neighborhood $N'$ of
$D_{p_0}$ in $M$ and an open neighborhood $\hat{U}$ of $X^i_{p_0}$ in $M$
such that for $J$ sufficiently small, $p \in J$ implies that $D_p \subset N' \subset \overline{N'} \subset (D_p)_\epsilon$, $X^i_p \subset \hat{U}$, and the forward orbit of any point $x \in \hat{U} - W^s(X^i_p)$ under $V_p$ enters $N'$ in finite time. Fig.~\ref{fig:tech1} illustrates an analogous claim for the case of a single fixed parameter value. From the claim made here, the main result can be shown as follows. Choose a subset $J' \subset J$ compact and connected with $p_0 \in \text{int }J'$. Let $D' = G^i_u(S^i_u \times J')$, the continuous image of a compact set, hence compact in $M \times J$. Note that, by definition of $G^i_u$, $D' = \sqcup_{p \in J'} W^u_{\text{loc}}(X^i_p)$. Let $D$ be the topological closure of $D'-f^{-1}(D')$ in $M \times J$.
Since $D'$ is contained in the local unstable manifold,
$f^{-1}|_{D'}$ is contracting. Hence, $f^{-1}(D') \subset D'$, so
$D \subset D'$. As $D$ is closed in $D'$ compact, $D$ is compact.  

Let $N'$ and $\hat{U}$ be as in the claim above. Then,
\begin{align*}
    D'-f^{-1}(D') = \sqcup_{p \in J'} \left(W^u_{\text{loc}}(X^i_p) - f^{-1}(W^u_{\text{loc}}(X^i_p))\right) 
    \subset \sqcup_{p \in J'} D_p 
    &\subset \overline{N'} \times J'.
\end{align*}
As $\overline{N'} \times J'$ is closed in $M \times J$, and since $D$ is the topological closure of $D'-f^{-1}(D')$ in $M \times J$, this implies that $D \subset \overline{N'} \times J'$.
Furthermore, $\sqcup_{p \in J'} D_p \subset D$ so $\overline{N'} \times J' = \sqcup_{p \in J'} \overline{N'} \subset \sqcup_{p \in J'} (D_p)_\epsilon \subset D_\epsilon$, which implies that $\overline{N'} \times J' \subset D_\epsilon$.
As $\overline{N'} \times J'$ is compact and disjoint from $\partial D_\epsilon$ which is closed in $M \times J$, there exists $r > 0$ such that $\overline{N'} \times J' \subset \left(\overline{N'} \times J'\right)_r \subset D_\epsilon$.
Let $N = \left(\overline{N'} \times J'\right)_r$. Then $N$ is open in $M \times J$ and satisfies $D \subset N \subset D_\epsilon$.
Let $U = \hat{U} \times \text{int }J'$.
Then $U$ is open in $M \times J$, $X^i_{p_0} \subset \hat{U} \times \{p_0\} \subset U$, and for every $(x,p) \in U - W^s(X^s_J)$, the forward orbit of $(x,p)$ under $V$ enters $N' \times \{p\} \subset N$ in finite time.
Thus, $\bigcup_{t \leq 0} \phi_t(N) \cup W^s(X^i_J)$ contains $U$, which completes the proof.


So, it suffices to prove the claim above.
We begin with the construction of the $C^1$ disk family for $f_{p_0}$ exactly
as in the proof of Lemma~\ref{lem:tech1}.
Then it is shown using the Inclination Lemma that for a $C^1$ perturbation
of the diffeomorphism $f_{p_0}$, constructing the $C^1$ disk family for the
perturbed
diffeomorphism gives a $C^1$ continuous disk family that is uniformly $C^1$-
close to the original $C^1$ continuous disk family.
Consequently, it is possible to choose $\hat{U}$ an open neighborhood of
$W^u_{\text{loc}}(X^i_{p_0})$ sufficiently small such that it is contained in the
perturbed disk family and, therefore, the forward orbit of each point in
$\hat{U}$ under the perturbed diffeomorphism either converges to the
perturbation of $X^i_{p_0}$ or enters $N'$ in finite time.
The full proof is provided in Appendix~\ref{ap:two}.
\end{proof}

The technical construction of Lemma~\ref{lem:tech2} is used to
show that the unstable manifold of any critical element
in $\partial W^s(X^s_J)$ must have nonempty
intersection with $\overline{W^s}(X^s_J) \cap M_{p_0}$.
By requiring that the intersection occurs in $M_{p_0}$, we will be able
to reduce to the fixed parameter case of Lemma~\ref{lem:int1},
which will ensure that
the unstable manifold actually intersects $W^s(X^s_J) \cap M_{p_0}$
(see Lemma~\ref{lem:int2} below).
Although Lemma~\ref{lem:bound1} and Lemma~\ref{lem:bound2} both
show the intersection of the unstable manifold with the closure of a
stable manifold, there is a crucial difference: for Lemma~\ref{lem:bound1}
this closure is taken in $M$ for a fixed parameter,
whereas for Lemma~\ref{lem:bound2} the closure is taken in $M \times J$.
As Example~\ref{ex:haus} showed, taking the closure in $M \times J$,
namely $\overline{W^s}(X^s_J)$, 
will in general give a larger set than taking the closure in $M$,
namely $\sqcup_{p \in J} \overline{W^s}(X^s_p)$.
This motivates the need for Lemma~\ref{lem:tech2} and Lemma~\ref{lem:bound2}
to explicitly treat the more difficult case where the closure is taken
in $M \times J$.

\begin{lemma}\label{lem:bound2}
For any $i \in I$,
$\{W^u(X^i_{p_0}) - X^i_{p_0}\} \cap \overline{W^s}(X^s_J)
\neq \emptyset$.
\end{lemma}

\begin{proof}[Proof of Lemma~\ref{lem:bound2}]
The proof is similar to that of Lemma~\ref{lem:bound1}, which
relied on the technical result of Lemma~\ref{lem:tech1} to show that
the distance between an annulus in $W^u_{\text{loc}}(X^i)$ (denoted by $D$
in that proof) and $W^s(X^s)$ was less than $\epsilon$ for any $\epsilon > 0$,
and then sent $\epsilon \to 0$ to establish the desired intersection.
Here, the goal is to use Lemma~\ref{lem:tech2} in a similar fashion.
The key difference is that, since the critical element $X^i_{p_0}$
lies in $\partial W^s(X^s_J)$, but not necessarily in $\partial W^s(X^s_{p_0})$, it is necessary to consider distances in parameter space $J$ as well.
In particular, Lemma~\ref{lem:tech2} is used to establish that there
exists a point $(x_r,p_r) \in W^s(X^s_J)$ such that, for any $\epsilon > 0$ sufficiently small, the distance
from an annulus in $W^u_{\text{loc}}(X^i_J)$ (denoted by $D$ in this proof)
to $(x_r,p_r)$ is less then $\epsilon$ and the distance from $p_0$ to $p_r$
is less than $r$, for $r > 0$ sufficiently small. Then, first sending $r \to 0$ and then sending $\epsilon \to 0$ results in a point in the desired intersection, which yields the main result. 
Let $B(p_0,r)$ be the ball of radius $r$ centered at $p_0$ in $J$.

Let $\epsilon > 0$.
Shrinking $\epsilon$ if necessary, by Lemma~\ref{lem:tech2} 
there exists a compact 
set $D \subset W^u(X^i_J) - X^i_J$
and an open neighborhood $N$ of $D$ in $M \times J$
such that $N \subset D_\epsilon$,
$D_\epsilon \cap X^i_J = \emptyset$,
and $\bigcup_{t \leq 0} \phi_t(N) \cup W^s(X^i_J)$
contains an open neighborhood of $X^i_{p_0}$ in $M \times J$ - call this
open neighborhood $N'$.
Let $N_r = N \cap \left(M \times B(p_0,r)\right)$
and $N'_r = N' \cap \left(M \times B(p_0,r)\right)$
be the intersections of the above neighborhoods with $M \times B(p_0,r)$.
Since $N'_r$ is an open neighborhood of 
$X^i_{p_0} \subset \partial W^s(X^s_J)$,
$N'_r \cap W^s(X^s_J) \neq \emptyset$.
Since $\bigcup_{t \leq 0} \phi_t(N_r) \cup W^s(X^i_J)$
contains $N'_r$, and because $N'_r \cap W^s(X^s_J) \neq \emptyset$
and $W^s(X^i_J) \cap W^s(X^s_J) = \emptyset$,
there exists $T > 0$ such that
$\phi_{-T}(N_r) \cap W^s(X^s_J) \neq \emptyset$.
By invariance of $W^s(X^s_J)$, this implies that
$N_r \cap W^s(X^s_J) \neq \emptyset$.
So, let $(x_r,p_r) \in N_r \cap W^s(X^s_J)$
and send $r$ to zero.
As $N_r \subset D_\epsilon \subset \overline{D}_\epsilon$ and $W^s(X^s_J) \subset \overline{W^s}(X^s_J)$,
$N_r \cap W^s(X^s_J) \subset \overline{D}_\epsilon \cap
\overline{W^s}(X^s_J)$ compact. Hence, passing to a subsequence if necessary we have that
$(x_r,p_r) \to (\hat{x}_\epsilon,\hat{p}_\epsilon) \in \overline{D}_\epsilon \cap \overline{W^s}(X^s_J)$.
By definition of $N_r$, since $r \to 0$ we must have $\hat{p}_\epsilon = p_0$, so for every $\epsilon > 0$ sufficiently small there exists $(\hat{x}_\epsilon,p_0) \in \overline{D}_\epsilon \cap \overline{W^s}(X^s_J)$. Fix some initial $\tilde{\epsilon} > 0$.  Then $\epsilon \leq \tilde{\epsilon}$ implies that $(\hat{x}_\epsilon,p_0) \in \overline{D}_{\tilde{\epsilon}} \cap \overline{W^s}(X^s_J)$ compact. So, sending $\epsilon \to 0$ and passing to a subsequence if necessary implies that $(\hat{x}_\epsilon,p_0) \to (x,p_0) \in \overline{D}_{\tilde{\epsilon}} \cap \overline{W^s}(X^s_J)$.
As $(\hat{x}_\epsilon,p_0) \in \overline{D}_\epsilon$,   $d_S((\hat{x}_\epsilon,p_0),D) \leq \epsilon$ for all $\epsilon > 0$ sufficiently small.
By continuity of $d_S$, $d_S((x,p_0),D) = \lim\limits_{\epsilon \to 0} d_S((\hat{x}_\epsilon,p_0),D) \leq \lim\limits_{\epsilon \to 0} \epsilon= 0$. Thus, $d_S((x,p_0),D) = 0$, so since $\{(x,p_0)\}$ and $D$ are compact, $(x,p_0) \in D$.
This implies that $(x,p_0) \in (D \cap M_{p_0})$.
By the above, $(x,p_0) \in \overline{W^s}(X^s_J)$ as well.
Thus, $\overline{W^s}(X^s_J) \cap (D \cap M_{p_0}) \neq \emptyset$. Since $D \cap M_{p_0} \subset W^u(X^i_{p_0}) - X^i_{p_0}$, the result follows.
\end{proof}


Thanks to the work of Lemma~\ref{lem:tech2} and Lemma~\ref{lem:bound2},
the varying parameter case treated in this section is effectively reduced to
the fixed parameter case of Section~\ref{sec:thm11}.
Hence, Lemma~\ref{lem:int2} is exactly analogous to its fixed parameter
counterpart Lemma~\ref{lem:int1} in both statement and proof.

\begin{lemma}\label{lem:int2}
For any $i \in I$, 
$W^u(X^i_{p_0}) \cap W^s(X^s_{p_0}) \neq \emptyset$.
\end{lemma}

\begin{proof}[Proof of Lemma~\ref{lem:int2}]
The proof is identical to the proof of Lemma~\ref{lem:int1},
substituting \\Lemma~\ref{lem:bound2} for Lemma~\ref{lem:bound1}.
\end{proof}

\begin{proof}[Proof of Theorem~\ref{thm:bound2}]

For any $p \in J$ and any $x \in \partial W^s(X^s_p)$,
there exist $x_n \in W^s(X^s_p)$ with $x_n \to x$.
Hence, $(x_n,p) \in W^s(X^s_J)$ with $(x_n,p) \to (x,p)$,
so $(x,p) \in \overline{W^s}(X^s_J)$ closed.
As $x \not\in W^s(X^s_p)$, this implies that
$(x,p) \in \partial W^s(X^s_J)$.
Hence,
\begin{align}
\sqcup_{p \in J} \partial W^s(X^s_p) \subset \partial W^s(X^s_J).
\label{eq:bound1}
\end{align}
We claim that for $J$ sufficiently small, for any $p \in J$ and $i \in I$,
$W^u(X^i_p) \cap W^s(X^s_p) \neq \emptyset$.
Let $i \in I$.  Then by Lemma~\ref{lem:int2}, we have that
$W^u(X^i_{p_0}) \cap W^s(X^s_{p_0}) \neq \emptyset$.
This implies that there exists $T > 0$ such that
$\phi_T(W^u_{\text{loc}}(X^i_{p_0})) \cap W^s_{\text{loc}}(X^s_{p_0})
\neq \emptyset$.
This intersection is trivially transverse since $W^s(X^s_{p_0})$ is an open set
in $M_{p_0}$.
Since $\{\phi_T(W^u_{\text{loc}}(X^i_p))\}_{p \in J}$ and
$\{W^s_{\text{loc}}(X^s_p)\}_{p \in J}$ are two $C^1$ continuous families over
$J$ of compact embedded submanifolds with boundary, and since they have
a point of transversal intersection at $p = p_0$, for $J$ sufficiently
small $p \in J$ implies \cite[Proposition A.3.16,Corollary A.3.18]{Ka99} that
$\phi_T(W^u_{\text{loc}}(X^i_p)) \cap W^s_{\text{loc}}(X^s_p) \neq \emptyset$.
Hence, for sufficiently small $J$ and
since $I$ is finite, $W^u(X^i_p) \cap W^s(X^s_p) \neq \emptyset$
for all $i \in I$ and $p \in J$, so the claim follows.
Let $(x,p) \in \partial W^s(X^s_J)$.
By Assumption~\ref{as:inf2}, $x \in W^s(X^i_p)$ for some $i \in I$. For this particular $i$, the claim implies that $W^u(X^i_p) \cap W^s(X^s_p) \neq \emptyset$.
Now the argument reduces to the fixed parameter case, and
we can use the proof of Theorem~\ref{thm:bound1} to show that
$W^s(X^i_p) \subset \partial W^s(X^s_p)$.
As $(x,p) \in \partial W^s(X^s_J)$ was arbitrary, we have
\begin{align}
\partial W^s(X^s_J) \subset 
\bigcup_{i \in I} W^s(X^i_J)
\subset \sqcup_{p \in J} \partial W^s(X^s_p).
\label{eq:bound2}
\end{align}
Then Eqs.~\ref{eq:bound1}-\ref{eq:bound2} imply the result.
\end{proof}



\begin{proof}[Proof of Corollary~\ref{cor:euc}]
Recall the definitions of $S^i_s$ and $F^i_s$ from the beginning
of Section~\ref{sec:thm21}.
Let $p_n \in J$ with $p_n \to p'$ for some $p' \in J$.
First, let $x \in \partial W^s(X^s_{p'})$.
Then $x \in W^s(X^i_{p'})$ for some $i \in I$.
So, there exists $T > 0$ such that $\phi_T(x,p') \in W^s_{\text{loc}}(X^i_{p'})$
and $y \in S^i_s$ such that $F^i_s(y,p') = \phi_T(x,p')$. Let $x_n = \phi_{-T}(F^i_s(y,p_n),p_n) \in W^s(X^i_{p_n})$ by invariance
of $W^s(X^i_{p_n})$.
Thus, $x_n \in W^s(X^i_{p_n}) \subset \partial W^s(X^s_{p_n})$ by
Theorem~\ref{thm:bound2}.
Furthermore, $x_n \to x$ since $\phi_{-T}$ and $F^i_s$ are $C^1$.
Hence, $x \in \liminf_{n \to \infty} \partial W^s(X^s_{p_n})$,
so $\partial W^s(X^s_{p'}) \subset
\liminf_{n \to \infty} \partial W^s(X^s_{p_n})$.

Next, let $x \in \limsup_{n \to \infty} \partial W^s(X^s_{p_n})$.
Then there exist a subsequence $\{p_{n_m}\}_{m=1}^\infty$ of
$\{p_n\}_{n=1}^\infty$ and a sequence $\{x_m\}_{m=1}^\infty$ such that
$x_m \in \partial W^s(X^s_{p_{n_m}})$ for all $m$ and $x_m \to x$.
By Theorem~\ref{thm:bound2},
$\partial W^s(X^s_{p_{n_m}}) \subset \partial W^s(X^s_J)$, so that
$(x_m,p_{n_m}) \in \partial W^s(X^s_J)$.
As $\partial W^s(X^s_J)$ is closed,
$\lim_{m \to \infty} (x_m,p_{n_m}) = (x,p') \in \partial W^s(X^s_J)$.
By Theorem~\ref{thm:bound2},
$\partial W^s(X^s_J) = \sqcup_{p \in J} \partial W^s(X^s_p)$,
so intersecting both sides with $M_{p'}$ implies that
$\partial W^s(X^s_J) \cap M_{p'} = \partial W^s(X^s_{p'}) \times \{p'\}$.
Hence, $(x,p') \in \partial W^s(X^s_J)$ implies that
$x \in \partial W^s(X^s_{p'})$.
Thus, $\limsup_{n \to \infty} \partial W^s(X^s_{p_n}) \subset
\partial W^s(X^s_{p'})$.
Together, these imply that $\lim_{n \to \infty} \partial W^s(X^s_{p_n}) = \partial W^s(X^s_{p'})$. As $p_n \to p'$ was arbitrary, this implies that $\{\partial W^s(X^s_p)\}_{p \in J}$ is a Chabauty
continuous family of subsets of $M$.
\end{proof}

\begin{proof}[Proof of Corollary~\ref{cor:comp}]
By Corollary~\ref{cor:euc}, we have that $\{\partial W^s(X^s_p)\}_{p\in J}$
is a Chabauty continuous family of subsets of $M$.
Since $M$ is compact, Hausdorff continuity is equivalent to
Chabauty continuity.
Hence, $\{\partial W^s(X^s_p)\}_{p\in J}$ is a Hausdorff continuous
family of subsets of~$M$.
\end{proof}

\begin{proof}[Proof of Corollary~\ref{cor:morse}]
Let $\Omega(V_{p_0})= \{X^i_{p_0}\}_{i=1}^n$ a finite union of hyperbolic 
critical elements since $V_{p_0}$ is Morse-Smale.
Palis showed \cite[Theorem~3.5]{Pa69}
that for any sufficiently small
$C^1$ perturbation to $V_{p_0}$, so for $J$ sufficiently small,
$p \in J$ implies that $V_p$ is still Morse-Smale with 
$\Omega(V_p) = \{X^i_p\}_{i=1}^n$.
Reorder the critical elements of $V_{p_0}$ if necessary such that $\{X^i_{p_0}\}_{i=1}^k = \Omega(V_{p_0}) \cap
\left(\partial W^s(X^s_J) \cap M_{p_0}\right)$, which is a finite union
of critical elements of $V_{p_0}$ since $\Omega(V_{p_0})$ is finite,
and $k \leq n$.
Note that $\{V_p\}_{p \in J}$ satisfies 
Assumption~\ref{as:hyp2} and Assumption~\ref{as:trans2} for $J$ sufficiently
small since $V_{p_0}$ is Morse-Smale.
Note that both $\bigcup_{i > k} X^i_{p_0}$ and
$\partial W^s(X^s_J) \cap M_{p_0}$ are compact, 
so since $M$ is a normal space there exists an open set $N$
such that $\partial W^s(X^s_J) \cap M_{p_0} \subset N$
and $N \cap \big( \bigcup_{i > k} X^i_{p_0} \big) = \emptyset$.
As $\Omega(V_{p_0}) = \bigcup_{i=1}^n X^i_{p_0}$, this implies that
$\Omega(V_{p_0}) \cap N = \bigcup_{i=1}^k X^i_{p_0} =
\Omega(V_{p_0}) \cap (\partial W^s(X^s_J) \cap M_{p_0})$.
Hence, Assumption~\ref{as:wand2} is satisfied.
So, it suffices to show that $\{V_p\}_{p \in J}$
satisfies Assumption~\ref{as:inf2} as well.

As in the proof of Theorem~\ref{thm:bound2}, $J$ sufficiently small implies that for every $i \in \{1,...,k\}$ and every $p \in J$, $W^u(X^i_p) \cap W^s(X^s_p) \neq \emptyset$. So, let $x \in W^u(X^i_p) \cap W^s(X^s_p)$.
As $x \in \overline{W^s}(X^s_p)$ closed and invariant,
the closure of the orbit of $x$ is contained in $\overline{W^s}(X^s_p)$.
Since $\alpha(x,p) = X^i_p$ is contained in the closure of the orbit of $x$,
$X^i_p \subset \overline{W^s}(X^s_p)$.
Since $X^i_p$ does not intersect $W^s(X^s_p)$,
$X^i_p\subset \partial W^s(X^s_p)$.
By Theorem~\ref{thm:bound2},
$\partial W^s(X^s_J) = \sqcup_{p \in J} \partial W^s(X^s_p)$,
so $X^i_p \subset \partial W^s(X^s_p)
\subset \partial W^s(X^s_J) \cap M_p$. Thus, for any $p \in J$, $\Omega(V_p) \cap \left(\partial W^s(X^s_J) \cap M_p\right) \supset \{X^i_p\}_{i=1}^k$.
Now, for any $i \in \{1,...,n\}$ and any $p \in J$, suppose that $X^i_p \subset \partial W^s(X^s_J)$. Then by Lemma~\ref{lem:int2},
$X^i_p \subset \partial W^s(X^s_J)$ implies that
$W^u(X^i_{p_0}) \cap W^s(X^s_{p_0}) \neq \emptyset$.
By the argument above, this implies that $X^i_{p_0} \subset \partial W^s(X^s_J) \cap M_{p_0}$.
Therefore, by definition of $k$ above, we must have
$i \in \{1,...,k\}$.
So, for any $p \in J$, as $\Omega(V_p) = \{X^i_p\}_{i=1}^n$, $\Omega(V_p) \cap \left(\partial W^s(X^s_J) \cap M_p\right) \subset \{X^i_p\}_{i=1}^k$. Combining this with the reverse inclusion above implies Assumption~\ref{as:inf2} is satisfied.
Thus, $\{V_p\}_{p \in J}$ satisfy
Assumptions~\ref{as:wand2}-\ref{as:trans2}.
Therefore, by Corollary~\ref{cor:comp},
$\{\partial W^s(X^s_p)\}_{p\in J}$ is a Hausdorff continuous
family of subsets of $M$.
\end{proof}

\section{Proof of Theorem~\ref{thm:time}}
\label{sec:thm22}

\begin{proof}[Proof of Theorem~\ref{thm:multi}]
First we show that $R$, $C$, and $\partial R$ are nonempty by connectedness
of any path from $p_1$ to $p_2$.
Then, we prove that every parameter value
$p^*$ in $\partial R$ is a boundary parameter value since we will see that
$y_{p^*} \in \partial W^s(X^s_J)$ which will imply, using
Theorem~\ref{thm:bound2}, that $p^* \in C$.
Next it is shown that $J_0$ is nonempty by noting that $C$ is closed in $\overline{J}$ compact, hence compact, and then arguing that there exists a point $\hat{p} \in C$ that achieves the minimum distance from $p_0$ to $C$, so that $d(p_0,\hat{p}) = d_S(p_0,C)$. 
Finally, we argue that
$J_0 = \{p \in \partial R:d(p_0,p) = d_S(p_0,\partial R)\}$
by choosing a minimal geodesic from $p_0$ to any fixed $p^* \in J_0$,
and arguing by connectedness that all points of the geodesic other than
$p^*$ must lie in $R$.

First we show that $R$, $C$, and $\partial R$ are nonempty.
Since $y_{p_1} \in W^s(X^s_{p_1})$, $p_1 \in R$ so $R$ is nonempty.
Let $\delta:[0,1] \to J$ be any continuous path in $J$ from $p_1$
to $p_2$, with $\delta(0) = p_1$ and $\delta(1) = p_2$.
Such a path exists because $J$ is a connected manifold, hence
pathwise connected.
As $y$ and $\delta$ are continuous and $[0,1]$ is connected,
$y_{\delta([0,1])}$ is connected.
Since $y_{\delta([0,1])}$ is connected and intersects
both $W^s(X^s_J)$ (at $y_{p_1}$) and $M \times J - W^s(X^s_J)$
(at $y_{p_2}$), it must intersect $\partial W^s(X^s_J)$.
Hence, there must exist $p^* \in \delta([0,1]) \subset J$
such that $y_{p^*} \in \partial W^s(X^s_J)$.
By Theorem~\ref{thm:bound2},
$\partial W^s(X^s_J) = \sqcup_{p \in J} \partial W^s(X^s_p)$.
Hence, $y_{p^*} \in \partial W^s(X^s_{p^*})$, so $p^* \in C$.
Thus, $C$ is nonempty.
As $C \cap R = \emptyset$, this implies that $\partial R$ is nonempty

Next we show that $\partial R \subset C$. Let $p^* \in \partial R$.
Then there exists a sequence $p_n \in R$ with $p_n \to p^*$.
Hence, by definition of $R$, $(y_{p_n},p_n) \in W^s(X^s_{p_n})$ for all $n$
with $(y_{p_n},p_n) \to (y_{p^*},p^*)$ since $y$ is $C^1$ and $p_n \to p^*$.
In particular, $(y_{p_n},p_n) \in W^s(X^s_{p_n}) \subset \overline{W^s}(X^s_J)$
for all $n$.
As $\overline{W^s}(X^s_J)$ is closed and $(y_{p_n},p_n) \to (y_{p^*},p^*)$,
this implies that $(y_{p^*},p^*) \in \overline{W^s}(X^s_J)$.
By Theorem~\ref{thm:bound2},
$\overline{W^s}(X^s_J) = \sqcup_{p \in J} \overline{W^s}(X^s_p)$.
Hence, $y_{p^*} \in \overline{W^s}(X^s_{p^*})$.
First assume towards a contradiction that $y_{p^*} \in W^s(X^s_{p^*})$.
Let $U$ be an open neighborhood of $X^s_{p^*}$ such that
$\overline{U} \subset \text{int }W^s_{\text{loc}}(X^s_{p^*})$.
Then there exists $T > 0$ such that $\phi_T(y_{p^*}) \in U$.
As $\text{int }W^s_{\text{loc}}(X^s_p)$ varies $C^1$ with parameter $p$,
there exists an open neighborhood $J'$ of $p^*$ in $J$ such that $p \in J'$
implies that $U \subset \text{int }W^s_{\text{loc}}(X^s_p)$.
As $U$ is open in $M$ and both $\phi_T$ and $y$ are $C^1$, shrinking
$J'$ if necessary implies that for $p \in J'$,
$\phi_T(y_p) \in U \subset \text{int }W^s_{\text{loc}}(X^s_p)$.
Hence, $J'$ is an open neighborhood of $p^*$ in $J$ such that
$J' \subset R$.
But this contradicts $p^* \in \partial R$.
So, since $y_{p^*} \in \overline{W^s}(X^s_{p^*})$
but $y_{p^*} \not\in W^s(X^s_{p^*})$, we must have $y_{p^*} \in \partial W^s(X^s_{p^*})$.
Hence, $p^* \in C$.

Fix $p_0 \in R$ and let $J_0$ be the set of boundary parameter
values $p^* \in C$ such that $J_0 = \{p^* \in C:d(p_0,p^*) = d_S(p_0,C)\}$.
We begin by showing that $J_0$ is nonempty. By Theorem~\ref{thm:bound2}, shrinking $J$ if necessary implies that $\sqcup_{p \in \overline{J}} \partial W^s(X^s_p) = \partial W^s(X^s_{\overline{J}})$. Thus, $y^{-1}\left(\partial W^s(X^s_{\overline{J}})\right) = \sqcup_{p \in \overline{J}} y_p^{-1}\left(\partial W^s(X^s_p)\right) = C$. As $y$ is continuous and $\partial W^s(X^s_{\overline{J}})$ is closed in $M \times \overline{J}$, $C$ is closed in $\overline{J}$. Since $J$ was chosen in Section~\ref{sec:time} such that $\overline{J}$ is compact, and $C$ is closed in $\overline{J}$, it follows that $C$ is compact.  Thus, since $C$ is compact and nonempty by the previous paragraph, and since $p_0$ is a point, there exists $p^* \in C$ such that $d(p_0,p^*) = d_S(p_0,C)$.  So, $p^* \in J_0$ which implies that $J_0$ is nonempty.

Finally, we show that
$J_0 = \{p \in \partial R:d(p_0,p) = d_S(p_0,\partial R)\}$.
Let $p^* \in J_0$. As $\overline{J}$ is convex, there exists $\gamma:[0,1] \to \overline{J}$, a minimal geodesic from $p_0$ to $p^*$, with $\gamma(0) = p_0$ and $\gamma(1) = p^*$, and the length of $\gamma$
is equal to $d(p_0,p^*)$.\footnote{For example,
if $J$ was a convex subset of Euclidean space then the image of
$\gamma$ would be the straight line segment between $p_0$ and $p^*$.}
For every $x \in [0,1)$, by definition of a minimal geodesic,
$d(p_0,\gamma(x)) < d(p_0,\gamma(1)) = d(p_0,p^*) = d_S(p_0,C)$,
where the last equality follows since $p^* \in J_0$.
This implies that for every $x \in [0,1)$, $\gamma(x) \not\in C$,
since otherwise we would have
$d_S(p_0,C) \leq d(p_0,\gamma(x)) < d(p_0,p^*)$, which would contradict
that $p^* \in J_0$ (so $d_S(p_0,C) = d(p_0,p^*)$).
Hence, $\gamma([0,1)) \cap C = \emptyset$.
Furthermore, since $[0,1)$ is connected and both $\gamma$ and $y$ are
continuous, $y_{\gamma([0,1))}$ is connected.
Assume towards a contradiction that there exists $x \in [0,1)$ such that
$y_{\gamma(x)} \not\in W^s(X^s_{\gamma(x)})$.
As $y_{p_0} \in W^s(X^s_J)$, $y_{\gamma(x)} \not\in W^s(X^s_J)$,
and $y_{p_0}, y_{\gamma(x)} \in y_{\gamma([0,1))}$ connected,
we must have $y_{\gamma([0,1))} \cap \partial W^s(X^s_J) \neq \emptyset$.  
So, there exists $x' \in [0,1)$ such that
$y_{\gamma(x')} \in \partial W^s(X^s_J)$.
By Theorem~\ref{thm:bound2},
$\partial W^s(X^s_J) = \sqcup_{p \in J} \partial W^s(X^s_p)$.
In particular, $y_{\gamma(x')} \in \partial W^s(X^s_{\gamma(x')})$.
But this implies $\gamma(x') \in C$, which contradicts $\gamma([0,1)) \cap C = \emptyset$.
So, we must have $y_{\gamma(x)} \in W^s(X^s_{\gamma(x)})$ for
all $x \in [0,1)$.
Hence, $\gamma([0,1)) \subset R$.
Let $p_n = \gamma\left(1 - \f{1}{n}\right)$.
Then $p_n \in R$ with $p_n \to p^*$, so $p^* \in \partial R$.
Because $\partial R \subset C$ as shown above,
$d_S(p_0,\partial R) \geq d_S(p_0,C) = d(p_0,p^*)$.
As $p^* \in \partial R$, it follows that $d_S(p_0,\partial R) \leq d(p_0,p^*)$.
Hence, combining these inequalities we have
$d_S(p_0,\partial R) = d(p_0,p^*)$.
As $p^* \in J_0$ was arbitrary, this implies
$J_0 = \{p^* \in \partial R:d(p_0,p^*) = d_S(p_0,\partial R)\}$.
\end{proof}

\begin{proof}[Proof of Corollary~\ref{cor:multi}]
Fix $p^* \in J_0$.
As $J_0 \subset C$, $p^* \in C$.
Hence, by definition of $C$, $y_{p^*} \in \partial W^s(X^s_{p^*})$.
By Theorem~\ref{thm:bound2},
$\partial W^s(X^s_{p^*}) = \bigcup_{i \in I} W^s(X^i_{p^*})$.
Thus, $y_{p^*} \in \partial W^s(X^s_{p^*})$ implies there exists a unique
$j \in I$ such that $y_{p^*} \in W^s(X^j_{p^*})$.
Let $X^*_J = X^j_J$ be the controlling critical element.
Then $y_{p^*} \in W^s(X^*_{p^*})$.
\end{proof}

\begin{lemma}\label{lem:defined}
Let $\gamma:[0,1] \to J$ be a path that satisfies Assumption~\ref{as:neigh2}
with embedded submanifold $N$.
For $p \in \gamma([0,1])$ let $T_p \subset [0,\infty)$ denote
the set of times $\{t \in [0,\infty):\phi_t(y_p) \in N\}$.
Then for $p \in \gamma([0,1))$, $T_p$ consists of a finite union of
closed intervals, so $\tau_N(p)$ is well-defined and finite.
For $p = \gamma(1)$, $T_p$ consists of a finite union of closed intervals
together with an interval of the form $[t',\infty)$ for some $t' > 0$,
so $\tau_N(p) = \infty$ is well-defined.
\end{lemma}

\begin{proof}[Proof of Lemma~\ref{lem:defined}]
First we show that the forward orbit of $y_p$ under $V_p$ is a one-dimen\-sional
$C^1$ embedded submanifold. This will imply, since this orbit is transverse to $\partial N$ and $N$, that its intersection with $N$ is a one-dimensional $C^1$ embedded submanifold
with boundary equal to its intersection with $\partial N$, which is a
zero-dimensional $C^1$ embedded submanifold.
By compactness, this manifold boundary consists of a finite number of points.
Then the formulas for $T_p$ are obtained
by considering the connected components of a one-dimensional manifold
in $[0,\infty)$.

First we show that for any $p \in \gamma([0,1])$ and $T>0$ sufficiently large such that $\phi_t(y_p) \notin \partial N$ for all $t \geq T$, then
$\phi(\cdot,y_p,p)^{-1}\left(\phi_{(0,T)}(y_p) \cap N\right)$ is
a one-dimensional embedded submanifold. The boundary of this submanifold is equal to
$\phi(\cdot,y_p,p)^{-1}\left(\phi_{(0,T)}(y_p) \cap \partial N\right)$,
which consists of a finite union of points.
So, let $p \in \gamma([0,1])$.
If $p = \gamma(1)$ then let $X_J = X^*_J$ where $X^*_J$ is the controlling
critical element corresponding to $\gamma(1)$ as in Corollary~\ref{cor:multi},
and choose $W^s_{\text{loc}}(X^*_{\gamma(1)})$ sufficiently small so that it is
contained in $N$.
Otherwise, let $X_J = X^s_J$ and choose $W^s_{\text{loc}}(X^s_J)$ sufficiently
small so that it is disjoint from $N$.
Then the orbit of $y_p$ under $V_p$ converges to $X_p$, so there exists $T > 0$
such that $\phi_T(y_p) \in \text{int }W^s_{\text{loc}}(X_p)$.
Hence, by definition of the local stable manifold, $t \geq T$ implies that
$\phi_t(y_p) \in \text{int }W^s_{\text{loc}}(X_p)$.
As $y_p \in W^s(X_p)$ but $y_p \not\in X_p$, the forward orbit of $y_p$
under $V_p$ does not contain any
critical elements, so $\phi(\cdot,y_p,p)$ is an injective $C^1$ immersion
from $[0,\infty)$ into $M$.
As $\phi(\cdot,y_p,p)$ is a continuous bijection onto its image,
$[0,T]$ is compact, and $M$ is Hausdorff,
$\phi(\cdot,y_p,p)$ is a homeomorphism from $[0,T]$ onto $\phi_{[0,T]}(y_p)$.
Hence, $\phi(\cdot,y_p,p)$ is a $C^1$ embedding from $[0,T]$ onto its image,
so $\phi_{(0,T)}(y_p)$ is a $C^1$ embedded submanifold in $M$
and $\phi(\cdot,y_p,p)$ is a $C^1$ diffeomorphism from $(0,T)$ onto
$\phi_{(0,T)}(y_p)$.

By Assumption~\ref{as:neigh2}, $\phi_{(0,T)}(y_p)$ is transverse to
$\partial N$, and it is trivially transverse to the interior of $N$ since the
dimension of $N$ is equal to the dimension of $M$.
Therefore, $\phi_{(0,T)}(y_p) \cap N$ is a one dimensional $C^1$ embedded
submanifold with boundary equal to $\phi_{(0,T)}(y_p) \cap \partial N$
a zero dimensional $C^1$ embedded submanifold.
Since $y_p, \phi_T(y_p) \not\in \partial N$,
$\phi_{(0,T)}(y_p) \cap \partial N = \phi_{[0,T]}(y_p) \cap \partial N$.
Furthermore, $N$ compact and $\phi_{[0,T]}(y_p)$ compact implies that
their intersection $\phi_{[0,T]}(y_p) \cap N$ is compact.
Therefore $\phi_{(0,T)}(y_p) \cap \partial N$ is a compact
zero dimensional \linebreak embedded submanifold.
As zero dimensional manifolds are discrete, this implies that \linebreak
$\phi_{(0,T)}(y_p) \cap \partial N$ consists of a finite union of points.
As $\phi(\cdot,y_p,p)$ is a $C^1$ diffeomorphism from $(0,T)$ onto
$\phi_{(0,T)}(y_p)$, it follows that
$\phi(\cdot,y_p,p)^{-1}\left(\phi_{(0,T)}(y_p) \cap N\right)$ is
a one-dimensional embedded submanifold with boundary equal to
$\phi(\cdot,y_p,p)^{-1}\left(\phi_{(0,T)}(y_p) \cap \partial N\right)$,
which consists of a finite union of points.

Next, we show that $\tau_N$ is well-defined and finite for
$p \in \gamma([0,1))$. Suppose $p \in \gamma([0,1))$.
Then $y_p, \phi_T(y_p) \not\in N$, so
$\phi_{(0,T)}(y_p) \cap N = \phi_{[0,T]}(y_p) \cap N$
is compact as $\phi_{[0,T]}(y_p)$ and $N$ are compact.
Thus, as $\phi(\cdot,y_p,p)$ is a $C^1$ diffeomorphism from $(0,T)$ onto
$\phi_{(0,T)}(y_p)$, this implies that
$\phi(\cdot,y_p,p)^{-1}\left(\phi_{(0,T)}(y_p) \cap N\right)$ is a compact
one-dimensional embedded submanifold in $(0,T)$ with boundary consisting
of a finite number of points.
Since it is a compact one-dimensional manifold, it has finitely many
connected components and each contains its manifold boundary.
Hence, $\phi(\cdot,y_p,p)^{-1}\left(\phi_{(0,T)}(y_p) \cap N\right)$
consists of a finite union of closed intervals.
As $\phi_t(y_p) \not\in N$ for all $t \geq T$, this implies that
$T_p = \{t \in [0,\infty): \phi_t(y_p) \in N\}$ consists of this finite
union of closed intervals.
So, $\tau_N(p) = \lambda(T_p)$, where $\lambda$ is the Lebesgue measure,
is well-defined and is equal to the sum of the lengths of all such intervals. This summation is finite since the intervals are contained in $[0,T]$ which has
finite length $T$.

Finally, we show that $\tau_N(\gamma(1)) = \infty$. Let $p = \gamma(1)$.
For $t \geq T$,
$\phi_T(y_p) \in \text{int }W^s_{\text{loc}}(X^*_{\gamma(1)}) \subset \text{int }N$,
so $\phi_t(y_p) \in \text{int } N$ for all $t \geq T$.
In particular, $\phi_t(y_p) \not\in \partial N$ for all $t \geq T$.
As $\phi(\cdot,y_p,p)^{-1}\left(\phi_{(0,T)}(y_p) \cap \partial N\right)$
consists of a finite number of points, let $t'$ be the largest value in
this set.
Then $t'$ represents the final intersection of the forward orbit of $y_p$
under $V_p$ with $\partial N$ since, by the above reasoning, no further intersections
occur for $t \geq T$.
We claim that for all $t \in [t',T]$, $\phi_t(y_p) \in N$.
Assume towards a contradiction that the claim is false.
Then there exists $\hat{t} \in (t',T)$ with $\phi_{\hat{t}}(y_p) \not\in N$.
As $\phi_{[\hat{t},T]}(y_p)$ is connected with $\phi_{\hat{t}}(y_p) \not\in N$,
$\phi_T(y_p) \in N$, and $N$ connected, there must exist
$t'' \in [\hat{t},T)$ such that $\phi_{t''}(y_p) \in \partial N$.
But, $t'' > t'$ with $\phi_{t''}(y_p) \in \partial N$,
so this contradicts that $t'$ was the final intersection of the forward orbit
of $y_p$ under $V_p$ with $\partial N$.
Hence, $[t',\infty) \subset T_p$.
As $t'$ is a manifold boundary point for   
$\phi(\cdot,y_p,p)^{-1}\left(\phi_{(0,T)}(y_p) \cap N\right)$,
there exists $\epsilon > 0$ such that
$[t'-\epsilon,t') \cap T_p = \emptyset$.
Hence, $\phi_{(0,t')}(y_p) \cap N = \phi_{[0,t'-\epsilon]}(y_p) \cap N$
is an intersection of two compact sets, hence compact.
Thus, $\phi(\cdot,y_p,p)^{-1}\left(\phi_{(0,t')}(y_p) \cap N\right)$
is a compact one-dimensional embedded submanifold in $(0,t')$ with
boundary consisting of a finite number of points.
Hence, $\phi(\cdot,y_p,p)^{-1}\left(\phi_{(0,t')}(y_p) \cap N\right)$
is a finite union of closed intervals.
Therefore, $T_p$ is the union of $[t',\infty)$ with a finite union of
closed intervals.
So, $\tau_N(p)$ is well-defined with $\tau_N(p) = \infty$.  
\end{proof}

\begin{lemma}\label{lem:limit}
Let $\gamma:[0,1] \to J$ be a path that satisfies Assumption~\ref{as:neigh2}
with embedded submanifold $N$.
Then $\lim\limits_{\substack{p \to \gamma(1) \\ p \in \gamma([0,1])}} \tau_N(p) = \infty$.
\end{lemma}

\begin{proof}[Proof of Lemma~\ref{lem:limit}]
By Lemma~\ref{lem:defined}, $\tau_N(\gamma(1)) = \infty$.
By the proof of Lemma~\ref{lem:defined}, there exists a final
time $t' \in [0,\infty)$ such that $\phi_{t'}(y_{\gamma(1)}) \in \partial N$, and
$t > t'$ implies that $\phi_t(y_{\gamma(1)}) \in \text{int }N$.
Choose any $\epsilon, K > 0$.
Then $\phi_{[t'+\epsilon,t'+\epsilon+K]}(y_{\gamma(1)}) \subset \text{int }N$.
As $\partial N$ and $\phi_{[t'+\epsilon,t'+\epsilon+K]}(y_{\gamma(1)})$ are compact and
disjoint in $M$ a normal space, there exists an open neighborhood
$U$ in $M$ such that
$\phi_{[t'+\epsilon,t'+\epsilon+K]}(y_{\gamma(1)}) \subset U \subset \text{int }N$.
As $U$ is open in $M$, and $\phi([t'+\epsilon,t'+\epsilon+K],y_p,p)$ is
compact and $C^1$ continuous with respect to $p$, then for $\delta > 0$
sufficiently small, $p \in \gamma((1-\delta,1))$ implies that
$\phi_{[t'+\epsilon,t'+\epsilon+K]}(y_p) \subset U \subset \text{int }N$.
So, for any $p \in \gamma((1-\delta,1))$, $[t'+\epsilon,t'+\epsilon+K] \subset T_p$.
By the proof of Lemma~\ref{lem:defined}, $p \in \gamma((1-\delta,1))$ implies that $T_p$ consists of a finite union closed intervals, and $\tau_N(p)$ is equal to the sum of the lengths of these intervals. Hence, $\tau_N(p)$ is at least as large as the length of the closed interval
that contains $[t'+\epsilon,t'+\epsilon+K]$, which is at least length $K$.
As $\tau_N(p) \geq K$ for all $p \in \gamma((1-\delta,1))$, $\tau_N(\gamma(1)) = \infty$, and $K > 0$ was
arbitrary, $\lim\limits_{\substack{p \in \gamma([0,1]) \\ p \to \gamma(1)}} \tau_N(p) = \infty$.
\end{proof}

\begin{figure}
\centering
\includegraphics[width=0.45\textwidth]{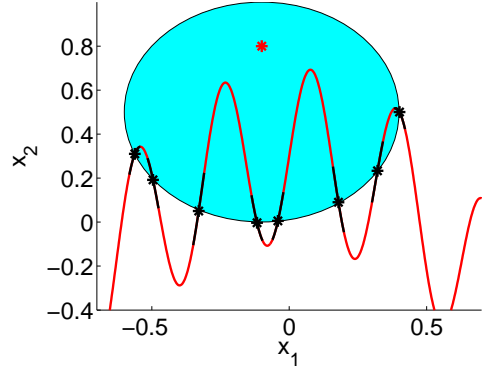}
\caption{For a fixed parameter value $p \in \gamma([0,1))$, the figure shows the intersection of the orbit of $y_p$ (red and black line segments) with the embedded
submanifold with boundary $N$ (cyan ellipse) containing an equilibrium point $X^*_p$ (red star). There are a finite number of intersections of the orbit of $y_p$
with $\partial N$ (black stars).
The orbit is a union of line segments of the form $\phi(T_i,y_p,p)$ (black),
which contain the intersection points, and line segments of the form
$\phi(T_i',y_p,p)$ (red), which are compact, contain no intersection points,
and intersect the black line segments on each end (although this
intersection is not visible in the figure).
This figure originally appeared in \cite{Fi17}.}
\label{fig:varies}
\end{figure}

\begin{lemma}\label{lem:cont}
Let $\gamma:[0,1] \to J$ be a path that satisfies Assumption~\ref{as:neigh2}
with embedded submanifold $N$.
Then $\tau_N$ is continuous over $\gamma([0,1))$
\end{lemma}

\begin{proof}[Proof of Lemma~\ref{lem:cont}]
Fix $s \in [0,1)$.
To show continuity of $\tau_N$ over $\gamma([0,1))$ it suffices to show that
it is continuous over a neighborhood of $\gamma(s)$ in $\gamma([0,1))$.
Let $\epsilon > 0$.
The proof proceeds by first showing that there exists $T > 0$ such that
$T_{\gamma(s')} \subset [0,T]$ for $s'$ close to $s$ since
$\phi_T(y_{\gamma(s')}) \in W^s_{\text{loc}}(X^s_{\gamma(s')})$.
Then, by stability of transversal intersections
and the implicit function theorem, 
it is shown that for every intersection point of the orbit
of $y_{\gamma(s)}$ under $V_{\gamma(s)}$ with $\partial N$,
$s'$ close to $s$ implies that
there exists a unique intersection point of the orbit
of $y_{\gamma(s')}$ under $V_{\gamma(s')}$ with $\partial N$ near the
original intersection point.
It is then argued that for $s'$ close to $s$, no new intersection points appear,
only perturbations of the original intersection points.
As $T_p$ is equal to a finite union of closed intervals whose boundaries
are equal to these intersection points by Lemma~\ref{lem:defined},
it will be shown that there is a one-to-one correspondence between
the closed intervals in $T_{\gamma(s)}$ and the closed intervals in
$T_{\gamma(s')}$.
As $\tau_N(\gamma(s'))$ is equal to the sum of the lengths of these (finitely many) closed intervals, and their lengths vary continuously
with parameter value since their endpoints (the intersection times) vary
continuously with parameter value, it will follow that
$|\tau_N(\gamma(s'))-\tau_N(\gamma(s))| < \epsilon$ for $s'$ close to $s$.
This proof is illustrated with the aid of Fig.~\ref{fig:varies}.

First we show that $T_p \subset [0,T)$ for some $T > 0$ and $p$ close to
$\gamma(s)$.
As $\gamma(s) \in R$, the forward orbit of $y_{\gamma(s)}$ under $V_{\gamma(s)}$
converges to $X^s_{\gamma(s)}$.  
So, there exists $T > 0$ such that
$\phi_T(y_{\gamma(s)}) \in \text{int }W^s_{\text{loc}}(X^s_{\gamma(s)})$.
As $W^s_{\text{loc}}(X^s_p)$ is open and varies $C^1$ with $p \in J$,
and as $\phi_T(y_p)$ varies $C^1$ with $p \in J$,
there exists $\delta > 0$ such that
$p \in \gamma((s-\delta,s+\delta))$ implies that
$\phi_T(y_p) \in \text{int }W^s_{\text{loc}}(X^s_p)$.
Similarly, choosing $W^s_{\text{loc}}(X^s_{\gamma(s)})$ sufficiently small
implies that for $p \in \gamma((s-\delta,s+\delta))$,
$W^s_{\text{loc}}(X^s_p)$ is disjoint from $N$.
Hence, $p \in \gamma((s-\delta,s+\delta))$ implies that for any
$t \geq T$, $\phi_t(y_p) \in \text{int }W^s_{\text{loc}}(X^s_p)$ which is
disjoint from $N$.
Therefore, $T_p \subset [0,T)$.

Next, persistence of the original intersection points is shown under
small changes in parameter values.
By Lemma~\ref{lem:defined}, there are a finite number of intersections of
$\phi_{[0,T]}(y_{\gamma(s)})$ with $\partial N$.
Let $t_i$ denote the $i$th intersection time of the orbit of $y_{\gamma(s)}$
under $V_{\gamma(s)}$ with $\partial N$.
Note that $\{\phi_{[0,T]}(y_p)\}_{p \in \gamma([0,1])}$ is a $C^1$ continuous
family of compact embedded submanifolds with boundary in $M$ with
$\phi_{[0,T]}(y_{\gamma(s)})$ transverse to $\partial N$ with finitely many points of intersection. Points of transversal intersection between compact embedded
submanifolds with boundary persist under $C^1$ perturbations
\cite[Proposition A.3.16]{Ka99}. Therefore, it follows by the implicit function theorem
that for $\delta > 0$ sufficiently small, there exist open neighborhoods
$T_i \subset [0,T]$ and $C^1$ functions
$k_i:\gamma((s-\delta,s+\delta)) \to T_i$ such that the following holds.
For $p \in \gamma((s-\delta,s+\delta))$,
$\phi(\cdot,y_p,p)^{-1}\left(\phi_{T_i}(y_p) \cap \partial N\right)
= \{k_i(p)\}$
and $k_i(\gamma(s)) = t_i$ for each $i$.
In other words, for each $p \in \gamma((s-\delta,s+\delta))$
and for each $i$, there exists a unique intersection of $\phi_{[0,T]}(y_p)$
with $\partial N$ that occurs in the time interval $T_i$.

It is shown next that for $\delta$ sufficiently small,
and for $p \in \gamma((s-\delta,s+\delta))$, the number of intersection
times of $\phi_{[0,T]}(y_p)$ with $\partial N$ is constant and each intersection
time varies continuously with the parameter values.
Let $m$ be the number of intersections of $\phi_{[0,T]}(y_{\gamma(s)})$ with
$\partial N$.
Let $T_i' \subset [0,T]$ be a connected closed interval for each $i \in \{1,...,m\}$ such that
$T_i' \cap T_{i-1} \neq \emptyset$, $T_i' \cap T_i \neq \emptyset$,
and $T_i'$ does not contain any times at which
$\phi_{[0,T]}(y_{\gamma(s)})$ intersects $\partial N$. For completeness,
we let $T_0 = \{0\}$ and $T_{m+1} = \{T\}$.
For each $i$, as $\phi_{T_i'}(y_{\gamma(s)})$ and $\partial N$ are
compact and disjoint, and since $\phi(T_i',y_p,p)$ varies $C^1$ with
respect to $p$, shrinking $\delta$ further if necessary implies that
for $p \in \gamma((s-\delta,s+\delta))$,
$\phi_{T_i'}(y_p)$ is disjoint from $\partial N$ for all $i$.
More specifically, $\phi_{T_i'}(y_p) \subset N$ if and only if
$\phi_{T_i'}(y_{\gamma(s)}) \subset N$.
We can write $[0,T] = \bigcup_{i=1}^m T_i \bigcup_{i=1}^{m+1} T_i'$,
as shown in Fig.~\ref{fig:varies}.
Hence, for $p \in \gamma((s-\delta,s+\delta))$,
the only intersection times of $\phi_{[0,T]}(y_p)$ with $\partial N$
occur in $\bigcup_{i=1}^m T_i$.
But, by the choice of the $T_i$ above, for $p \in \gamma((s-\delta,s+\delta))$
this implies that the only intersection times of
$\phi_{[0,T]}(y_p)$ with $\partial N$ are $\bigcup_{i=1}^m k_i(p)$.
By the choice of $T$ above, this implies that
for $p \in \gamma((s-\delta,s+\delta))$,
the only intersection times of the orbit of $y_p$ under $V_p$ with $\partial N$
are $\bigcup_{i=1}^m k_i(p)$.
Hence, for $p \in \gamma((s-\delta,s+\delta))$, the number of
intersections of the orbit of $y_p$ under $V_p$ with $\partial N$ is constant.

Finally, we show that for $p \in \gamma((s-\delta,s+\delta))$,
there is a one-to-one correspondence between the closed intervals in
$T_p$ and the closed intervals in $T_{\gamma(s)}$, where the interval lengths
can be brought arbitrarily close to each other for sufficiently small $\delta$.
We will conclude that $|\tau_N(p)-\tau_N(\gamma(s))| < \epsilon$.
By Lemma~\ref{lem:defined}, as $(s-\delta,s+\delta) \subset [0,1)$,
for any $p \in \gamma((s-\delta,s+\delta))$,
$T_p$ consists of a finite union of closed intervals whose boundary points
are the intersection times.
Then for each $i$, $[k_{i-1}(p),k_i(p)] \subset T_p$ if and only if
$\phi_{T_i'}(y_p) \subset N$ since $T_i' \subset [k_{i-1}(p),k_i(p)]$,
$\phi_{[k_{i-1}(p),k_i(p)]}(y_p)$ is connected, and there are no intersections
of the orbit of $y_p$ under $V_p$ with $\partial N$ in the time interval
$(k_{i-1}(p),k_i(p))$.
Hence, for each $i$, 
$[k_{i-1}(p),k_i(p)] \subset T_p$ if and only if
$\phi_{T_i'}(y_p) \subset N$ if and only if
$\phi_{T_i'}(y_{\gamma(s)}) \subset N$
if and only if $[k_{i-1}(\gamma(s)),k_i(\gamma(s))] \subset T_{\gamma(s)}$.
Therefore, since the number of intersections is constant
over $p \in \gamma((s-\delta,s+\delta))$,
$T_p$ and $T_{\gamma(s)}$ consist of the same finite number of
corresponding closed intervals which differ only slightly in their endpoints,
the intersection times $\{k_i(p)\}_{i=1}^m$ and $\{k_i(\gamma(s)\}_{i=1}^m$, respectively.
Shrink $\delta$ such that for $p \in \gamma((s-\delta,s+\delta))$
and each $i$, $|k_i(p)-t_i| < \f{\epsilon}{2m}$, where $t_i = k_i(\gamma(s))$.
Since $T_p$ consists of the same number of corresponding
closed intervals as $T_{\gamma(s)}$, 
each closed interval in $T_p$ has length within
$\f{\epsilon}{m}$ of the length of the corresponding
interval in $T_{\gamma(s)}$.
For $p \in \gamma((s-\delta,s+\delta))$,
as $\tau_N(p)$ is equal to the sum of the lengths of the closed intervals
in $T_p$, there are $m$ such closed intervals in $T_p$, and the length
of each closed interval in $T_p$ is within $\f{\epsilon}{m}$
of the corresponding interval in $T_{\gamma(s)}$,
$|\tau_N(p)-\tau_N(\gamma(s))| < \epsilon$.
Hence, $\tau_N$ is continuous at $\gamma(s)$.
\end{proof}

\begin{proof}[Proof of Theorem~\ref{thm:time}]
Fix $p^* \in J_0$ and let $\gamma:[0,1] \to J$ be a $C^1$ path satisfying
Assumption~\ref{as:neigh2} and such that $\gamma(0) = p_0$, $\gamma(1) = p^*$,
and $\gamma([0,1)) \subset R$.
By Lemma~\ref{lem:defined}, $\tau_N:\gamma([0,1]) \to [0,\infty]$
is well-defined and $\tau_N(\gamma(1)) = \infty$.
By Lemma~\ref{lem:limit},
$\lim\limits_{\substack{p \in \gamma([0,1]) \\ p \to p^*}} \tau_N(p) = \infty = \tau_N(\gamma(1))$,
so $\tau_N$ is continuous at $p^* = \gamma(1)$.
By Lemma~\ref{lem:cont}, $\tau_N$ is continuous over $\gamma([0,1))$.
Hence, $\tau_N$ is continuous over $\gamma([0,1])$.
\end{proof}

\section{Conclusion}\label{sec:conc}

This work considers a weakly $C^1$ continuous family of vector fields on Euclidean space or on a compact Riemannian manifold. It shows that if the family possesses a stable equilibrium point, and if the vector field along its region of attraction (RoA) boundary satisfies Morse-Smale-like assumptions, then the RoA boundary is Hausdorff continuous (for a compact Riemannian manifold) or Chabauty continuous (for Euclidean space) with respect to parameter. This result builds on a decomposition of the RoA boundary into the union of the stable manifolds of the critical elements it contains. Furthermore, it is shown that this decomposition persists under small variations in parameter values. A recent complement to this work \cite{Fi18b} shows that the assumptions of this paper can be relaxed to Morse-Smale-like along with generic assumptions about a vector field at a single initial parameter value, so that it is not necessary to assume that no new nonwandering points can enter the RoA boundary under parameter perturbations.

These technical results are used to provide theoretical motivation for algorithms which numerically determine the recovery set $R$ by computing parameter values at points on its boundary $\partial R$. 
The algorithms proceed by identifying a controlling critical
element in the RoA boundary, and varying parameter values so as to maximize
the time spent by the trajectory in a neighborhood of that controlling critical
element. It is shown that the time spent by the trajectory in this neighborhood is
continuous with respect to parameter values, and approaches infinity as the
parameter values approach $\partial R$, thereby justifying the algorithmic
approach. Recently developed algorithms \cite{Fi18c} for numerically computing
boundary parameter values do not require prior knowledge of the
controlling critical element. Theoretical motivation of those algorithms again builds on the results developed in this paper.

\section{Acknowledgments}

Discussions with Ralf Spatzier and Wouter van Limbeek were very helpful
in developing the ideas in this paper. The authors gratefully acknowledge the contribution of the U.S.~National Science Foundation through grant ECCS-1810144.

\bibliographystyle{siamplain}
\bibliography{siam_ref}

\appendix

\section{Proof of Lemma~\ref{lem:tech1}}\label{ap:one}

Let $f$ and $D$ be defined as in the proof outline of Lemma~\ref{lem:tech1}. The proof is simple in the case that $W^s(X^i)$ has dimension zero. In that case, let $\hat{U} = W^u_{\text{loc}}(X^i)$. Then the interior of $\hat{U}$ is the open neighborhood that satisfies the claim of Lemma~\ref{lem:tech1}, so we may assume that the dimension of $W^s(X^i)$ is greater than zero.

We begin by constructing a $C^1$ continuous disk family along $D$ using the vector field $V$. This disk family will be extended to all of $W^u_{\text{loc}}(X^i)$ using the diffeomorphism $f$.
Let $A = \partial W^u_{\text{loc}}(X^i)$.
Then $A$ is a $C^1$ immersed submanifold of $W^u_{\text{loc}}(X^i)$ of codimension
one, and $\bigcup_{t < 0} \phi_t(A) = W^u_{\text{loc}}(X^i)$.
As the time-$t$ flow restricted to $W^u_{\text{loc}}(X^i)$ is a contraction for
any $t < 0$, for each $y \in D$ there exists a unique $x = x(y) \in A$
and $t = t(y) \geq 0$ such that $\phi(t(y),y) = x(y) \in A$.
By the tubular neighborhood theorem \cite[Theorem~6.24]{Lee13},
as discussed above in Section~\ref{sec:defs}, there exists a $C^1$ continuous family of pairwise disjoint disks $\{\tilde{D}(x)\}_{x \in A}$ centered along $A$ and transverse to $W^u(X^i)$. For each $y \in D$, let
$\tilde{D}(y) = \phi_{-t(y)}(\tilde{D}(x(y)))$.
Since $A$ is a $C^1$ immersed submanifold of $W^u(X^i)$ of codimension one,
there exists a real vector-valued function
$s$ defined on a neighborhood of $W^u_{\text{loc}}(X^i)$ in $W^u(X^i)$ such that
$s$ is a $C^1$ submersion and $A = s^{-1}(0)$. Then for any $x \in A$, $T_xA$ is equal to the kernel of $ds_x$ \cite[Proposition~5.38]{Lee13}.
Let $y \in D$ and choose $t$ such that $\phi(t,y) \in A$.
Then $s \circ \phi(t,y) = 0$.
Furthermore, $\pd{}{t} \big( s \circ \phi(t,y) \big) = ds_{\phi(t,y)}V(\phi(t,y))$.
Since the time-$t$ flow (for $t<0$) restricted to $W^u_{\text{loc}}(X^i)$ is
a contraction, $V$ is transverse to $A$, so $V_{\phi(t,y)} \not\in T_{\phi(t,y)}A$.
Thus, $V_{\phi(t,y)}$ is not in the kernel of $ds_{\phi(t,y)}$, so
$\pd{}{t} \big( s \circ \phi(t,y) \big) = ds_{\phi(t,y)}V(\phi(t,y)) \neq 0$.
Hence, by the implicit function theorem there exists a neighborhood
$N'$ of $y$ in $W^u_{\text{loc}}(X^i)$ and a $C^1$ function $t:N' \to \mathbb{R}$
such that $y' \in N'$ implies that $s \circ \phi(t(y'),y') = 0$ or,
equivalently, $\phi(t(y'),y') \in A$.
Thus, the function $t = t(y)$ is $C^1$, and $x = x(y) = \phi(t(y),y)$
is also $C^1$ since $\phi$ and $t(y)$ are $C^1$.
Therefore, by construction, the disk family
$\{\tilde{D}(y)\}_{y \in D}$ is $C^1$ continuous. As $D$ is compact, we may shrink $W^u_{\text{loc}}(X^i)$, shrink the disk family $\{\tilde{D}(x)\}_{x \in A}$, and choose $N$ an open neighborhood in $M$ such that $D \subset N \subset \overline{N} \subset D_\epsilon$
and for each $y \in D$, $\tilde{D}(y) \subset N$.

Next the $C^1$ continuous disk family along $W^u_{\text{loc}}(X^i)$ is
constructed for $f$ by backward iteration of the disk family
above, and it is shown to contain an open neighborhood of $X^i$.
If $X^i$ is an equilibrium point, let $x_0 = X^i$ and let $S = M$.
If $X^i$ is a periodic orbit, let $x_0 \in X^i$ such that $f$ is the
first return map for a cross section $S$ centered at $x_0$. It is possible to choose local $C^1$ coordinates in a neighborhood of $x_0$ such that, in these coordinates, $W^s_{\text{loc}}(x_0)$ is equal to $\mathbb{R}^s \times \{0\}$ and $W^u_{\text{loc}}(x_0)$ is equal to $\{0\} \times \mathbb{R}^u$, where $s+u=n$.  In particular,
let $Y$ be a neighborhood of $x_0$ in $S$ such that $Y$ can be expressed in the $C^1$ local coordinates of \cite[p.~80-81]{Pa82}, which aligns the local stable manifold with $\mathbb{R}^s \times \{0\}$ and the local unstable manifold with $\{0\} \times \mathbb{R}^u$.
In these coordinates, $Y = W^s_{\text{loc}}(x_0) \times W^u_{\text{loc}}(x_0)$, and we have $Y \subset \mathbb{R}^n = E^s \oplus E^u$, where $E^s$ and $E^u$ are the stable and unstable eigenspaces of $df_{x_0}$, respectively, and $W^s_{\text{loc}}(x_0) \subset E^s$, $W^u_{\text{loc}}(x_0) \subset E^u$. Next we extend the disk family $\{\tilde{D}(x)\}_{x \in D}$ to all of $W^u_{\text{loc}}(x_0)$. For each $x \in W^u_{\text{loc}}(x_0)- \{x_0\}$, let $l(x)$ be the smallest integer $l \geq 0$
such that $f^{l(x)}(x) \in D$. Such an intersection always exists by the construction of $D$ so $l(x)$ is finite. For $x \in W^u_{\text{loc}}(x_0)- \{x_0\}$, 
let $\tilde{D}^0(x) = \tilde{D}(f^{l(x)}(x))$, and for each $m \in 0,1,2, ..., l(x)$, let
$\tilde{D}^m(x)$ be the connected component of $f^{-1}(\tilde{D}^{m-1}(x)) \cap Y$ that
contains $f^{l(x)-m}(x)$, where we set $f^0(x) = x$.
Then let $\tilde{D}(x) = \tilde{D}^{l(x)}(x)$ and $\tilde{D}(x_0) = W^s_{\text{loc}}(x_0)$.
This gives a family of connected $C^1$ disks centered along $W^u_{\text{loc}}(x_0)$.

Next we show that the family $\{\tilde{D}(x)\}_{x \in W^u_{\text{loc}}(x_0)}$ is a $C^1$ continuous disk family. Let $y \in W^u_{\text{loc}}(x_0) - \{x_0\}$ and recall that $f^{l(y)}(y) \in D$. First we claim that there exists an open neighborhood $\overline{N}$ of $f^{l(y)}(y)$ in $W^u_{\text{loc}}(x_0) - \{x_0\}$ such that $\{\tilde{D}(x)\}_{x \in \overline{N}}$ is a $C^1$ continuous disk family.  If $f^{l(y)}(y) \in \text{int }D$ then let $\overline{N}$ be an open connected neighborhood of $f^{l(y)}(y)$ in $\text{int }D$.  As $\{\tilde{D}(x)\}_{x \in D}$ is a $C^1$ continuous disk family by construction and $\overline{N} \subset D$, it is clear that $\{\tilde{D}(x)\}_{x \in \overline{N}}$ is a $C^1$ continuous disk family as well.  So, suppose $f^{l(y)}(y) \in \partial D$.
Let $B_O$ and $B_I$ be the outer and inner topological boundaries of $D$, respectively. Then $\partial D = B_O \cup B_I$ and, by construction of $D$, $f(B_I) = B_O$.  First suppose $f^{l(y)}(y) \in B_O$. Note that by definition of $l(y)$ this implies that $f^{l(y)}(y) = y \in B_O$, with $l(y)=0$, since $f^{l(y)}(y) \in B_O$ implies that $f^{l(y)-1}(y) = f^{-1}(f^{l(y)}(y)) \in B_I \subset D$. Let $\overline{N}$ be an open connected neighborhood of $f^{l(y)}(y) = y$ in $D$. Then $\{\tilde{D}(x)\}_{x \in \overline{N}}$ is a $C^1$ continuous disk family since $\overline{N} \subset D$ and $\{\tilde{D}(x)\}_{x \in D}$ is a $C^1$ continuous disk family.  So, it suffices to consider $f^{l(y)}(y) \in B_I$.  Let $\overline{N}$ be a connected open neighborhood of $f^{l(y)}(y)$ in $W^u_{\text{loc}}(x_0) - \{x_0\}$.  Let $N_1 = \overline{N} \cap D$ and let $N_2 = \overline{N} \cap \left(W^u_{\text{loc}}(x_0) - \text{int }D\right)$.  Then $\overline{N} = N_1 \cup N_2$ and $N_1 \cap N_2 = \overline{N} \cap B_I$.  As $f(B_I) = B_O$ and $f$ is continuous, shrinking $\overline{N}$ if necessary implies that $f(N_2) \subset D$.  Then $\{\tilde{D}(x)\}_{x \in N_1}$ is a $C^1$ continuous disk family and $\{\tilde{D}(x)\}_{x \in f(N_2)}$ is a $C^1$ continuous disk family since $N_1, f(N_2) \subset D$ and $\{\tilde{D}(x)\}_{x \in D}$ is a $C^1$ continuous disk family. For $x' \in M$ and $\tilde{S} \subset M$, let $\pi_{x'}(\tilde{S})$ denote projection onto the connected component of $\tilde{S}$ containing $x'$ (noting that $\pi_{x'}(\tilde{S}) = \emptyset$ if $x' \not\in \tilde{S}$).
Note that for any $x \in N_2$, if $x \not\in D$ then $\tilde{D}(x) = \pi_x\big( f^{-1}(\tilde{D}(f(x))) \cap Y \big)$, and if $x \in D$ then $x \in B_I$ so $\tilde{D}(x) = f^{-1}(\tilde{D}(f(x)))$ by construction. Therefore, since $f^{-1}$ is a $C^1$ diffeomorphism, $\{\tilde{D}(x)\}_{x \in N_2} = \big\{\pi_{f^{-1}(x')}\big(f^{-1}(\tilde{D}(x')) \cap Y\big)\big\}_{x' \in f(N_2)}$ is a $C^1$ continuous disk family. 
Since $N_1 \cap N_2 = \overline{N} \cap B_I$, and
$f^{-1}(\tilde{D}(x)) = \tilde{D}(f^{-1}(x))$ for $x \in B_O$, the $C^1$ disk families $\{\tilde{D}(x)\}_{x \in N_1}$ and $\{\tilde{D}(x)\}_{x \in N_2}$ agree along their intersection. Hence, as each is a $C^1$ continuous disk family and they agree along their intersection, $\{\tilde{D}(x)\}_{x \in \overline{N}} = \{\tilde{D}(x)\}_{x \in N_1} \cup \{\tilde{D}(x)\}_{x \in N_2}$ is a $C^1$ continuous disk family.

Thus, for any $y \in W^u_{\text{loc}}(x_0) - \{x_0\}$, there exists an open neighborhood $\overline{N}$ of $f^{l(y)}(y)$ in $W^u_{\text{loc}}(x_0) - \{x_0\}$ such that $\{\tilde{D}(x)\}_{x \in \overline{N}}$ is a $C^1$ continuous disk family.  We claim that $f^{-l(y)}(\overline{N})$ is an open neighborhood of $y$ in $W^u_{\text{loc}}(x_0)-\{x_0\}$ such that $\{\tilde{D}(x)\}_{x \in f^{-l(y)}(\overline{N})}$ is a $C^1$ continuous disk family. To prove this, we will observe that $\{\tilde{D}(x)\}_{x \in f^{-m}(\overline{N})}$ is a $C^1$ continuous disk family for each $m \in \{0, 1,...,l(y)\}$. The proof proceeds by induction on $m$, and note that we have already proven the result for $m=0$. So, assume that $\{\tilde{D}(x)\}_{x \in f^{-m}(\overline{N})}$ is a $C^1$ continuous disk family for some $m \in \{0,1,...,l(y)-1\}$. Note that for any $x \in f^{-(m+1)}(\overline{N})$, $\tilde{D}(x) = \pi_x\big(f^{-1}(\tilde{D}(f(x))) \cap Y\big)$. Thus, since $f^{-1}$ is a $C^1$ diffeomorphism, $\{\tilde{D}(x)\}_{x \in f^{-(m+1)}(\overline{N})} = \big\{\pi_{f^{-1}(x')}\big(f^{-1}(\tilde{D}(x')) \cap Y\big)\big\}_{x' \in f^{-m}(\overline{N})}$ is a $C^1$ continuous disk family. As every point $y \in W^u_{\text{loc}}(x_0) - \{x_0\}$ has an open neighborhood, call it $N'$, such that $\{\tilde{D}(x)\}_{x \in N'}$ is a $C^1$ continuous disk family, $\{\tilde{D}(x)\}_{x \in W^u_{\text{loc}}(x_0) - \{x_0\}}$ is a $C^1$ continuous disk family.

It remains to consider continuity at $x_0$, where $\tilde{D}(x_0) = W^s_{\text{loc}}(x_0)$.
Since the family $\{\tilde{D}(x)\}_{x \in D}$ is $C^1$ continuous and transverse
to $W^u_{\text{loc}}(x_0)$, by the Inclination Lemma \cite{Pa69}, the family $\{\tilde{D}(f^{-m}(x))\}_{x \in D}$ converges uniformly in the $C^1$ topology to $W^s(x_0)$ as $m \to \infty$. By construction, for every $x \in W^u_{\text{loc}}(x_0) - \{x_0\}$, $\tilde{D}(x)$ is obtained by taking the connected component of $f^{-1}(\tilde{D}(f(x)) \cap Y$ that contains $x$.  Therefore, since $W^s(x_0) \cap Y = W^s_{\text{loc}}(x_0)$,
$\{\tilde{D}(x)\}_{x \in W^u_{\text{loc}}(x_0) - \{x_0\}}$ converges 
to $W^s_{\text{loc}}(x_0)$ as $x \to x_0$.
This shows that the family $\{\tilde{D}(x)\}_{x \in W^u_{\text{loc}}(x_0)}$ defined above is
$C^1$ continuous at $x_0$. Therefore, $\{\tilde{D}(x)\}_{x \in W^u_{\text{loc}}(x_0)}$ is a $C^1$ continuous disk family.

Thus, $\{\tilde{D}(x)\}_{x \in W^u_{\text{loc}}(x_0)}$ is a $C^1$
continuous family of disks transverse to $W^u_{\text{loc}}(x_0)$ and such that $x,y \in W^u_{\text{loc}}(x_0)$ with $x \neq y$ implies that
$\tilde{D}(x) \cap \tilde{D}(y) = \emptyset$.
Hence, there exists a $C^1$ injective function
$F:B_r^s \times W^u_{\text{loc}}(x_0) \to M$, where $B_r^s$ is the closed ball of radius $r > 0$ centered at the origin in $\mathbb{R}^s$, such that $F(B_r^s \times \{x\}) = \tilde{D}(x)$ for every $x \in W^u_{\text{loc}}(x_0)$.
Thus, since $F$ is a continuous injection between manifolds of the same
dimension, by invariance of domain \cite[Theorem~2B.3]{Ha01},
$\hat{U} = F(\text{int }B_r^s \times \text{int }W^u_{\text{loc}}(x_0))$
is an open neighborhood of $\text{int }W^u_{\text{loc}}(x_0)$ in $M$.
By construction, for every $x \in \hat{U} - W^s_{\text{loc}}(x_0)$,
its forward orbit intersects $\tilde{D}(y)$ for some $y \in D$ in finite positive time.
But by construction, $\tilde{D}(y) \subset N$ for all $y \in D$.
Hence, $\bigcup_{t < 0} \phi_t(N) \cup W^s_{\text{loc}}(x_0)$ contains
$\hat{U}$ an open neighborhood of $X^i$ in $M$.

\section{Proof of Lemma~\ref{lem:tech2}}\label{ap:two}


First note that $W^s_{\text{loc}}(X^i_p)$, $W^u_{\text{loc}}(X^i_p)$,
the tubular neighborhoods centered along them, and the flows are $C^1$ continuous with respect to parameter value, so that $p$ close to $p_0$ implies that the disk family $\{\tilde{D}(x)\}_{x \in D_p}$ of the perturbed vector field will be uniformly $C^1$-close to that of the original. Hence, $(D_p)_\epsilon$ will be $C^1$-close to $(D_{p_0})_\epsilon$. So, $J$ sufficiently small implies that there exists $N'$ open such that for $p \in J$, $D_p \subset N' \subset \overline{N'} \subset (D_p)_\epsilon$. Let $x_0(p)$ be defined in the natural way so that it is $C^1$ with respect to parameter, $x_0(p) = X^i_p$ if $X^i_p$ is an equilibrium point, and if $X^i_p$ is a periodic orbit then $x_0(p) \in X^i_p$. Construct the full disk family $\{\tilde{D}(x)\}_{x \in W^u_{\text{loc}}(x_0(p))}$ for the
perturbed vector field as in the proof of Lemma~\ref{lem:tech1}.

In the proof of the Inclination Lemma, a diffeomorphism $g$ sufficiently $C^1$-close to $f$ implies that the same bounds on $g$ and its partial derivatives will hold as for $f$. Hence, the uniform bounds on the inclinations (slopes) obtained for the transverse disk family $\{\tilde{D}(x)\}_{x \in D_{p_0}}$ can also be taken to apply to $\{\tilde{D}(x)\}_{x \in D_p}$ for $p$ sufficiently close to $p_0$.
In particular, for any $\delta > 0$ there exists $Z_N > 0$ such that $l \geq Z_N$ implies
that for any $x \in D_{p_0}$, $\tilde{D}(f_{p_0}^{-l}(x))$ is $\delta$ $C^1$-close to $W^s_{\text{loc}}(x_0(p_0))$ and for any $x' \in D_p$, $\tilde{D}(f_p^{-l}(x'))$ is $\delta$ $C^1$-close to $W^s_{\text{loc}}(x_0(p))$. Furthermore, for $p$ sufficiently close to $p_0$, $W^s_{\text{loc}}(x_0(p))$ is $\delta$ $C^1$-close to $W^s_{\text{loc}}(x_0(p_0))$.  Therefore, for every $x \in D_{p_0}$ and $x' \in D_p$, and for every $l \geq Z_N$, by the triangle inequality the distance from $\tilde{D}(f^{-l}_{p_0}(x))$ to $\tilde{D}(f^{-l}_p(x^\prime))$ is no more than the distance from $\tilde{D}(f^{-l}_{p_0}(x))$ to $W^s_{\text{loc}}(x(p_0))$ plus the distance from $W^s_{\text{loc}}(x(p_0))$ to $W^s_{\text{loc}}(x(p))$ plus the distance from $W^s_{\text{loc}}(x(p))$ to $\tilde{D}(f_p^{-l}(x'))$. Hence, $\tilde{D}(f_{p_0}^{-l}(x))$ is $3 \delta$ $C^1$-close to $\tilde{D}(f_p^{-l}(x'))$. As $\{\tilde{D}(x)\}_{x \in D_{p_0}}$ and $\{\tilde{D}(x)\}_{x \in D_p}$ are uniformly $C^1$-close and $Z_N$ is finite, for $p$ sufficiently close to $p_0$ we have that the two disk families $\{\tilde{D}(f_{p_0}^{-l}(x))\}_{x \in D_{p_0}}$ and $\{\tilde{D}(f_p^{-l}(x))\}_{x \in D_p}$ are uniformly $3 \delta$ $C^1$-close for $l \leq Z_N$. Hence, combining the above, we have that the disk families $\{\tilde{D}(x)\}_{x \in W^u_{\text{loc}}(x_0(p_0))}$ and $\{\tilde{D}(x)\}_{x \in W^u_{\text{loc}}(x_0(p))}$ are uniformly $3 \delta$ $C^1$-close.

Constructing $F_p$ analogously to the construction of $F$ in Lemma~\ref{lem:tech1}, this implies that $F_p(\text{int }B_r^s \times \text{int }W^u_{\text{loc}}(x_0(p)))$ has Hausdorff distance no greater than $3 \delta$ from $F_{p_0}(\text{int }B_r^s \times \text{int }W^u_{\text{loc}}(x_0(p_0)))$. In particular, this implies that $\delta > 0$ and $\hat{U}$ can be chosen sufficiently small such that for $J$ sufficiently small, $p \in J$ implies that $F_p(\text{int }B_r^s \times \text{int }W^u_{\text{loc}}(x_0(p))) \supset \hat{U}$. Hence, for every $x \in \hat{U} - W^s_{\text{loc}}(x_0(p))$, the forward orbit of $x$ intersects $N'$ in finite time.

\end{document}